\documentclass[draft]{amsart}
\usepackage{mathrsfs,enumerate,amssymb,verbatim,fancyhdr}
\usepackage[bordercolor=red,linecolor=red,backgroundcolor=none]{todonotes}
\usepackage[bordercolor=red,linecolor=red,backgroundcolor=none]{todonotes}
\usepackage{bm}
\usepackage{tikz-cd}

 \numberwithin{equation}{section}

\title{$BV$-packing integral in $\rn$}
\author{Krist\'yna Kuncov\'a}
\address{Department of Mathematical Analysis, Charles University,
So\-ko\-lovsk\'a 83, 186~00 Prague 8, Czech Republic}
\email{\tt kuncova@karlin.mff.cuni.cz}

\def\beq{\begin{equation}\aligned}
\def\eeq{\endaligned\end{equation}}

\def\cl{\operatorname{cl}}
\def\inter{\operatorname{int}}

\def\tr{\operatorname{tr}}
\def\d{\,\mathrm{d}}

\def\en{\mathbb N}
\def\ep{\varepsilon}
\def\er{\mathbb R}

\def\ff{\varphi}

\def\charfn_#1{\chi_{\strut #1}}
\def\diam{\operatorname{diam}}

\def\div{\operatorname{div}}
\def\Div{\operatorname{Div}}
\def\ndiv{\mathfrak{div}\,}

\def\loc{{{\text{loc}}}}

\def\rn{{\er^n}}
\def\rm{{\er^m}}
\def\sys#1{(#1)}

\def\qak{\mathfrak Q}

\def\cB{\mathcal B}
\def\C{\mathcal C}

\def\F{\mathcal F}

\def\TT{\mathfrak T}
\def\G{\mathcal G}
\def\H{\mathscr H}
\def\I{\mathcal I}
\def\K{\mathcal K}
\def\L{\mathscr L}

\def\N{\mathcal N}
\def\P{\mathscr P}
\def\Q{\mathscr Q}
\def\R{\mathcal R}
\def\S{\mathcal S}
\def\T{\mathcal T}

\def\V{\mathcal V}

\def\TT{\mathfrak T}
\def\GR{\mathcal {GR}}

\def\loc{\operatorname{loc}}

\def\Lip{\operatorname{Lip}}

\def\tR{{\boldsymbol{R}}}
\def\tpR{\overline{\boldsymbol{R}}}
\def\tbv{{\boldsymbol{\cBV}}}

\def\BV{{\operatorname{\it BV}}}
\def\cBV{{\cB\V}}

\def\vint{-\kern-11pt\int}
\def\vvint{-\kern-9pt\int}

\def\eqn#1$$#2$${\begin{equation}\label#1#2\end{equation}}

\def\cl{\operatorname{cl}}

\def\en{\mathbb N}
\def\ep{\varepsilon}
\def\er{\mathbb R}
\def\ff{\varphi}

\def\charfn_#1{\chi_{\strut #1}}
\def\diam{\operatorname{diam}}

\def\div{\operatorname{div}}
\def\Div{\operatorname{Div}}

\def\loc{\text{loc}}

\def\rn{\er^n}

\def\qak{\mathfrak Q}
\def\A{\mathcal A}

\def\C{\mathcal C}

\def\F{\mathcal F}
\def\G{\mathscr G}
\def\H{\mathscr H}
\def\I{\mathcal I}

\def\L{\mathscr L}

\def\N{\mathcal N}
\def\P{\mathscr P}
\def\R{\mathscr R}
\def\S{\mathscr S}
\def\T{\mathscr T}

\def\V{\mathcal V}

\def\loc{\operatorname{loc}}

\def\Lip{\operatorname{Lip}}

\def\vint{-\kern-11pt\int}

\def\eqn#1$$#2$${\begin{equation}\label#1#2\end{equation}}

\def\Hn{\H^n}
\def\Hs{\H^s}
\def\Hnmj{\H^{n-1}}

\def\dx{\mathrm{d}x}

\def\dt{\mathrm{d}t}

\def\inn{\operatorname{in}}

\def\XXint#1#2#3{{\setbox0=\hbox{$#1{#2#3}{\int}$ }
\vcenter{\hbox{$#2#3$ }}\kern-.6\wd0}}

\def\Hsdelta{\H^s_\delta}

\def\ext{\operatorname{ext}}
\def\intt{\operatorname{int}}

\def\cl{\operatorname{cl}}

\def\L{\mathcal L}
\def\eqn#1$$#2$${\begin{equation}\label#1#2\end{equation}}
\def\F{\mathcal F}
\def\R{\mathcal R}
\def\div{\operatorname{div}}
\def\Div{\operatorname{Div}}
\def\rn{\mathbb R^n}

\newtheorem{theorem}{Theorem}[section]
\newtheorem{lemma}[theorem]{Lemma}
\newtheorem{proposition}[theorem]{Proposition}
\newtheorem{corollary}[theorem]{Corollary}
\theoremstyle{definition}
\newtheorem{definition}[theorem]{Definition}
\newtheorem{notation}[theorem]{Notation}

\newtheorem{remark}[theorem]{Remark}
\newtheorem{example}[theorem]{Example}
\newtheorem{examples}[theorem]{Examples}
\long\def\red#1\endred{{\color{red}#1}}
\long\def\green#1\endgreen{{\color{green}#1}}
\long\def\blue#1\endblue{{\color{blue}#1}}
\long\def\magenta#1\endmagenta{{\color{magenta}#1}}

\begin{document}

\newpage

\begin{abstract}
We introduce new integrals (called packing $\R$ and $\R^*$ integrals)
which combine advantages of integrals developed by Pfeffer \cite{Pf}, Mal\'y \cite{malydistr},
Kuncov\'a and Mal\'y \cite{KM} and Mal\'y and Pfeffer \cite{mpf}. We prove Gauss-Green theorem
in generality of the new integrals and provide comparison with the integrals
mentioned above and some others (like $MC_{\alpha}$ by Ball and Preiss \cite{dball}).
\end{abstract}

\maketitle
\tableofcontents
\section{Introduction}

The Gauss-Green divergence theorem
\eqn{gg}
$$
\int_{A}\div \mathbf u(x)\,dx= \int_{\partial_*A}\mathbf u\cdot\boldsymbol\nu_A\,d\Hnmj
$$
holds whenever $A\subset\rn$ is a bounded $BV$ set (or, in another terminology, a 
bounded set of finite perimeter) and $\mathbf u\in C^1(\rn,\rn)$.
Here, $\partial_*A$ is the essential boundary and $\boldsymbol\nu_A$ is
the measure-theoretic unit exterior  normal. This setting 
and its history can be found e.g.\ in \cite{AFP}.
If we want to allow discontinuous derivatives, routine approximation 
arguments give \eqref{gg} if $\mathbf u\in C(\rn,\rn)$ and 
$\div\mathbf u(x)\in L^1(\rn)$. 
Beyond Lebesgue integrability of $\div\mathbf u(x)$,
a natural idea is to consider the divergence in the sense of distributions.
Particularly deep results have been obtained for divergence measure vector
fields,
see e.g.\ Chen, Torres and Ziemer \cite{ZiemGG}, Ziemer \cite{ZiemC} or
\v{S}ilhav\'{y} \cite{Silh1, Silh2, Silh3}.

We pursue another direction. If $\mathbf u$ is differentiable,
the divergence formula still holds even if the divergence 
is not Lebesgue integrable. This phenomenon indicates that the $L^1$ 
setting is not the ultimate generality if we want to consider the divergence
as 
\textit{a pointwise function}. 
Such 
a
divergence 
still plays the 
role of divergence in the sense of distributions, but the task is to what
extent non-absolutely integrable pointwise functions can be 
represented as distributions.
The problem exists already in the one-dimensional case where it has been
solved by the Denjoy-Perron integral. The multidimensional case
has been treated by many authors, among the most important 
contribution we mention \cite{JKS,Ma1,He}.
The most important progress in this direction
has been done by Pfeffer \cite{Pf}, who developed a theory which 
can be used for the divergence theorem on $BV$ sets. In his setting,
indefinite integral is a function on $BV$ sets, so that the definite 
integral on the left of \eqref{gg} is the evaluation of the indefinite 
integral at $A$.
An interesting extension has been introduced by Pfeffer and Mal\'{y} in \cite{mpf}.
Their effort leads to  the $\R^*$ integral, which is stable under reasonable operations and 
has a rich family of integrable functions. In particular, the $\R^*$ integral includes Pfeffer's $\R$ integral \cite{Pf}
and the $1$-dimensional Henstock-Kurzweil integral.

In a series of papers \cite{KM,malydistr,ph}, 
a new non-absolutely convergent 
integral with respect to distributions, called \textit{packing 
integral}, has been introduced. Since main 
motivation comes from the divergence theorem and related results again,
it is natural to ask on comparison of this integral with Pfeffer's 
approach. In its original setting, the indefinite packing integral is a 
functional on smooth (or Lipschitz) test functions and its evaluation
at $BV$ sets does not make sense. Therefore, the definite integral
on the left of \eqref{gg} is the evaluation of the indefinite integral
of $\chi_A \div \mathbf u$ at a test function which is $1$ on a neighborhood of 
$\partial A$. 

The Pfeffer integral (one of the equivalent versions) is based on 
Riemann-type sums 
$$
\sum_{i=1}^m\big|\F(E_i)-f(x_i)\L(E_i)\big|
$$
where $E_i\subset\rn$ are disjointed
$BV$ sets,  $x_i\in\rn$ are tags, $\L$ is Lebesgue measure
 and $\F$ is the candidate for the 
indefinite integral.
In our setting, we also use sums
\eqn{oursums}
$$
\sum_{i=1}^m q_{x_i,r_i}(\F-f(x_i)\L)
$$
where $(q_{x,r})_{x,r}$ is a system of suitable seminorms.

Let $A\subset\rn$ be a bounded $BV$ set.
Suppose that $\mathbf u\in C(\rn,\rn)$ and the indefinite packing integral of a function $f$ is 
the flux of $\mathbf u$, so that $f=\Div \mathbf u$ in a general sense.
 We would be happy to see that
$$
\int_{A}\Div \mathbf u(x)\,dx= \int_{\partial_*A}\mathbf u\cdot\boldsymbol\nu\,d\Hnmj,
$$
where the integral on the left means the integration of $f\chi_A$.
(In other words, the characteristic function of $A$ acts as a multiplier
for the integration of $f$.)
However, in the setting of \cite{KM} it is not clear
how to estimate the sums \eqref{oursums}
(and it is probably impossible without additional hypotheses).
It helps if we can omit $x_i$ belonging to a small set, say of 
$\sigma$-finite $\Hnmj$ Hausdorff measure, namely, just $\partial_*A$.
This change of definition requires the indefinite integral to be 
a \textit{charge}, a functional on $BV\cap L^{\infty}$ functions continuous
with respect to a convergence specified below. Charges can be 
represented as functions on $BV$ sets, and by this series of thoughts 
we recover most ingredients of Pfeffer's setting.

In this paper we present modifications of the packing integral
which contains Pfeffer's $\R$ integral and Pfeffer's and Mal\'{y}'s $\R^*$ integral. We 
apply the new integrals to obtain more general versions of the divergence theorem.
In the end we discuss the relationships between particular integrals 
including the one-dimensional Henstock-Kurzweil-Stieltjes integral and 
$MC_\alpha$ integral.

\section{Notation and Preliminaries}

\begin{notation}
Let $E$ be a subset of $\rn$. Then $d(E)$ denotes the diameter of $E$, i.e.
$$
d(E)=\sup\{|y-x|; x,y\in E \}.
$$

Let $x\in\rn$ and $r>0$. Then $B(x,r)$ denotes the open ball
$$
B(x,r)=\{y\in\rn; |y-x|<r\}
$$
and $\bar{B}(x,r)$ denotes the closed ball
$$
\bar{B}(x,r)=\{y\in\rn; |y-x|\leq r\}.
$$

The Lebesgue measure of $E$ is denoted by $|E|$ or $\L(E)$. 
\end{notation}
\begin{definition}
We say, that measurable sets $A$ and $B$ are \emph{equivalent} (or $A$ and $B$ belong to the same \emph{equivalence class}) 
if $|A\triangle B|=0$, where $A\triangle B$ denotes the symmetric difference of the sets $A$ and $B$.
\end{definition}
\begin{definition}
Let $s\geq 0$.
The \emph{$s$-dimensional outer Hausdorff measure} of a set $E\subset\rn$ 
is defined as $\Hs(E)=\lim_{\delta\to0+} \H^s_\delta(E)$, where
$$\H^s_\delta(E)=\inf\left\{
\sum_{i=1}^\infty \alpha_s \left(
\frac{\diam(C_i)}{2}
\right)^s; C_i \subset\rn,
E\subset \bigcup_{i=1}^\infty C_i, \diam(C_i)<\delta
\right\}
$$
and $\alpha_s=\frac{\pi^{\frac{s}{2}}}{\Gamma\left(\frac{s}{2}+1\right)}$.
\end{definition}

\begin{proposition}\label{l:Pflipeo}
Let $A\subset \rn$ be a set and let $\varphi:A\to\rn$ be a Lipschitz mapping. Then
$\Hnmj(\varphi(A))\leq (\Lip \varphi)^{n-1}\Hnmj(A)$.
\end{proposition}
\begin{proof}
For the proof and further details see \cite[Section 2.4.1]{EG}.
\end{proof}

\begin{definition}
Let $A\subset \rn$ be a measurable set and let $x\in\rn$. Then  we define the \emph{lower density} of $A$ at $x$ as
$$
\underline{\Theta}
(A,x)
:=\liminf_{r\to0+} \frac{|A\cap {B}(x,r)|}{|{B}(x,r)|}
$$
and the \emph{upper density} of $A$ at $x$ as
$$
\overline{\Theta}
(A,x)
:=\limsup_{r\to0+} \frac{|A\cap {B}(x,r)|}{|{B}(x,r)|}
.
$$
The \emph{essential closure $\cl_*A$}, \emph{essential interior $\inter_*A$} and \emph{essential boundary $\partial_* A$} 
are then defined as
$$
\cl_*A=\{x\in\rn; \overline{\Theta}(A,x)>0\},
$$
$$
\inter_*A=\{x\in\rn; \underline{\Theta}(A,x)=1\}
$$
and
$$
\partial_*A=\cl_*A\setminus \inter_*A.
$$
\end{definition}

\begin{definition}
We say that a measurable set $A\subset\rn$ is \emph{admissible} if 
$\inter_* A\subset A\subset \cl_* A$.
\end{definition}
\begin{remark}
Our definition of admissible set differs from that used by Mal\'{y} and Pfeffer
in \cite{mpf}, according to which $\partial A$ is required to be compact.
\end{remark}

\begin{remark}
Let $A$, $A'$ be measurable sets such that $|A\triangle A'|=0$.
Then $\cl_*A=\cl_*A'$, $\intt_* A=\intt_* A'$ and $\partial_* A=\partial_* A'$.

Hence, for every bounded measurable set $A$ we can find an admissible set $A'$ such that $|A\triangle A'|=0$.
\end{remark}

\section{$\BV$ sets and charges}

In this section we will present some basic facts about spaces of sets of bounded variation ($\BV$ sets) and about charges
which  will be essential in further definitions.
For details see \cite{Pf}, \cite{npf} and \cite{dpw}.

\begin{definition}
Let $U\subset\rn$ be an open set. For a measurable set $E\subset\rn$
we define the \emph{perimeter of $E$ in $U$} as 
$$
P(E,U)=\sup
\Big\{\int_{U\cap E}\div \ff\colon \ff\in C_c^1(U),\;\|\ff\|_{\infty}\le 1\Big\}.
$$
If $P(E,U)<\infty$, then the distributional gradient $D\chi_E$ of $\chi_E$
in $U$ is a vector-valued Radon measure and 
$P(E,U)$ is exactly its total variation. By the De Giorgi--Federer theorem,
we can compute $P(E,U)$ as
$$
P(E,U)=\H^{n-1}(\partial_*E\cap U).
$$
The particular choice $U=\rn$ gives the 
\emph{perimeter} of $E$
$$
P(E)=\|E\|=\H^{n-1}(\partial_*E).
$$
If $A\subset \rn$ is just measurable, we define 
also the \emph{relative perimeter of $E$ in $A$} as
$$
P(E,\inn A)=\H^{n-1}(\partial_*E\cap\inter_* A).
$$
There is a distinction between $P(E,\inn U)$ and $P(E,U)$ if $U$
is open, see Example \ref{ex:perim} below.

We say that a measurable set $E$ is a \emph{locally $\BV$ set}, if $P(E,A)<\infty$ for each bounded open set $A$.
A measurable set $E$ is called a \emph{$\BV$ set}, if $|E|+\|E\|<\infty$.

The family of all $\BV$ sets and all locally $\BV$ sets is denoted by $\cBV$ and $\cBV_{\loc}$, respectively.
The family of all bounded $\BV$ sets  is denoted by $\tbv$.
\end{definition}

\begin{example}\label{ex:perim}
Let $E=B(0,1)$ and $A=B(0,2)\setminus\{x\in\er^2:|x|=1\}$ be subsets of $\er^2$.
Then $P(E,A)=\Hnmj(\emptyset)=0$, whereas $P(E,\inn A)=\Hnmj\big(\partial\left(B\left(0,1\right)\right)\big)=2\pi$.
\end{example}

\begin{remark}
If $n=1$,
each $BV$ set $E$ is equivalent to a set $\bigcup_{i=1}^k (a_i,b_i)$, where $a_1<b_1<\cdots<a_k<b_k$
are real numbers.
In this case, $\|E\|=2k$.
\end{remark}

\begin{definition}
Let $A$ be a locally $\BV$ set. Then we define the \emph{critical boundary} of $A$ as
$$
\partial_cA=\left\{
x\in\rn; \limsup_{r\to 0+} \frac{P(A,B(x,r))}{r^{n-1}}>0
\right\}.
$$
The \emph{critical interior} $\intt_c A$ and \emph{critical exterior} $\ext_c A$ are then defined as
$$
\intt_c A=\intt_* A\setminus \partial_c A,\qquad \ext_c A=\ext_*A \setminus \partial_c A.
$$
\end{definition}

In the following, we will define the regularity of a $BV$ set. 
This concept has been first introduced by Kurzweil, Mawhin and Pfeffer in \cite{mawhin}.
In this article, we use the modification established by Pfeffer in \cite{npf}.

\begin{definition}
Let $E\subset\rn$ be a bounded $\BV$ set and let $x\in\rn$. 
The \emph{regularity} of the set $E$ is the number
$$
r(E) = 
\begin{cases}
\frac{|E|}{d(E)\|E\|}  & \text{if } |E| >0,\\
0& \text{if } |E| = 0.\\
\end{cases}
$$

The \emph{regularity} of the pair $(E,x)$ is the number
$$
r(E,x) = 
r(E\cup\{x\})=
\begin{cases}
\frac{|E|}{d(E\cup\{x\})\|E\|}  & \text{if } |E| >0,\\
0& \text{if } |E| = 0.\\
\end{cases}
$$

Let $\varepsilon>0$. We say that the set $E$ and the pair $(E,x)$ are \emph{$\varepsilon$-regular} if $r(E)>\varepsilon$
and $r(E,x)>\varepsilon$, respectively.
A system $P=\{(A_1,x_1),\ldots, (A_m,x_m)\}$, $A_i\subset \rn$ and $x_i\in \rn$, is called
$\varepsilon$-regular if $r(A_i,x_i)>\varepsilon$ for $i=1,\ldots,m$.
 
Let us note that every $\varepsilon$-regular $BV$ set is bounded.
\end{definition}

\begin{remark}
For every bounded $\BV$ set $E$  we have the estimate
$r(E)\leq 1/(2n)$. 
Especially, the regularity of a ball is equal to $1/(2n)$ 
(see \cite[Chapter 2.3]{Pf}).
\end{remark}

\begin{definition}
A \emph{dyadic cube} is an interval
$$
\prod_{i=1}^n \left[ \frac{k_i}{2^m},\frac{k_i+1}{2^m}\right],
$$
where $m,k_1,\ldots,k_n$ are integers.
A dyadic cube $C'$ is called the \emph{mother} of a dyadic cube  $C$ if $C'$ is the smallest (with respect to inclusion)
dyadic cube properly containing $C$.

A finite (possibly empty) union of  nondegenerate compact intervals in $\rn$
is called a \emph{figure}.
A \emph{dyadic figure} is a figure that is a union of finitely many dyadic cubes.
\end{definition}

\begin{definition}
Let $B$ be a bounded $\BV$ set.
We say that a sequence $\{B_i\}\subset\tbv$  
\emph{converges} to $B$ in $\tbv$ if
\begin{enumerate}
\item $\bigcup_{i=1}^\infty B_i$ is a bounded set,
\item $\lim_{i\to\infty} |B_i\triangle B|=0$ and $\sup_i \|B_i\|<\infty$. 
\end{enumerate}
\end{definition}

\begin{lemma}\label{l:aproxkrychle}
Let $A$ be a bounded $\BV$ set. Then there exists a sequence $\{A_i\}$ of dyadic figures which
converges to $ A$ in $\tbv$.
\end{lemma}
\begin{proof}
See \cite[Proposition 1.10.3]{Pf}.
\end{proof}

\begin{definition}
We say that a function $\F:\tbv\to\er$ is a \emph{charge} if $\F$ satisfies the following conditions:
\begin{enumerate}
\item $\F(A\cup B)=\F(A)+\F(B)$ for each disjoint bounded $\BV$ sets $A$ and $B$.
\item Given $\varepsilon>0$ there exists an $\eta>0$ such that
$|\F(C)|<\varepsilon$ for each $\BV$ set $C\subset B(0,1/\varepsilon)$ with $\|C\|<1/\varepsilon$ and $|C|<\eta$.
\end{enumerate}
\end{definition}

\begin{remark}
Let $E$ be a bounded $\BV$ set and $\F$ be a charge. Since $\F$ is additive and vanishes on bounded negligible sets,
 $\F(E)$ depends only on the equivalence class of the set $E$.
\end{remark}

\begin{notation}
Let $E$ be a locally $BV$ set and $\F$ be a charge. Then $\F\lfloor_E$ denotes the charge
$\F\lfloor_E(A):=\F(A\cap E)$, $A\in\tbv$.
\end{notation}
\begin{definition}
Let $E$ be a locally $\BV$ set and let $\F$ be a charge.
We say that $\F$ is a \emph{charge in $E$}
if $\F=\F\lfloor_E$. 
\end{definition}

\begin{proposition}\label{l:charchar}
An additive function $\F$ on $\tbv$ is a charge if and only if 
either of the following conditions is satisfied.
\begin{enumerate}
\item 
For given $\varepsilon$ there is a $\theta>0$ such that for every $\BV$ set $B\subset B(1/\varepsilon)$ we have
$$
|\F(B)|<\theta|B|+\varepsilon(\|B\|+1).
$$
\item
$\lim \F(A_i)=0$ for each sequence $\{A_i\}$ with $A_i\to\emptyset$ in $\tbv$ .
\end{enumerate}
\end{proposition}
\begin{proof}
See \cite[Proposition 2.2.6, Proposition 2.1.2]{Pf}.
\end{proof}

\begin{definition}
Let $A$ be a locally $\BV$ set.
We say that an additive function $\F:\tbv\to\er$ is a \emph{flux in $A$} of a vector field 
$\mathbf u\in C(\bar{A},\rn)$, if for each $E\in\tbv$ we have
$$
\F(E)=
\int_{\partial_* (E\cap A)} \mathbf u\cdot\boldsymbol\nu_{E\cap A}\d\Hnmj,
$$
where $\boldsymbol\nu_{E\cap A}$ denotes the unit exterior normal of $E\cap A$.

In the case $A=\rn$ we say that $\F$ is just a \emph{flux} of $\mathbf u$.
\end{definition}

\begin{examples}\label{r:ident}
\begin{enumerate}
\item 
Let $n=1$.
Since every bounded set $E\subset \er$ is equivalent to a finite disjoint union of compact intervals $\bigcup_{i=1}^k [a_i,b_i]$, 
each additive function $\F$ on $\tbv$ can be written as
$$\F(E)
=\sum_{i=1}^k \left(u(b_i)-u(a_i)\right),$$
where $u:\er\to\er$.
The additive function $\F$ is a charge if and only if $u$ is continuous (see \cite[Remark 2.1.5]{Pf}).
In other words, $\F$ can be represented as the distributional derivative of a continuous function $u$.

\item
Let $\F$ be a flux in $A$ of a continuous vector field 
$\mathbf u\in C(\bar{A},\rn)$, where $A$ is a locally $\BV$ set. 
Then $\F$ is a charge (see \cite[Example 2.1.4]{Pf}).
On the other hand, a charge needs not to be of this form. 
For an example see  \cite[Example 2.1.10]{Pf}.

\item
Let $f\in\L_{\loc}^1 (\rn)$ be a function.
Then the function $\F:\tbv\to\er$ defined as 
$$
\F(A)=\int_{A} f\d\L
$$
is a charge. (See \cite[Example 2.1.3]{Pf}.)

\item Let $A$ be a measurable set with $\Hnmj(A)>0$.
Then the function $\F:\tbv\to\er$ defined as $\F(E)=\Hnmj(E\cap A)$ is not a charge.

Without loss of generality we may assume $A$ to be bounded.
At first let us suppose that $\Hnmj(A)<\infty$. Then there is a constant $c$ such that
for every $k\in\en$ we can find a sequence of balls $\{B_i\}$ 
with $A\subset \bigcup_{i=1}^\infty B_i$, $\diam B_i<1/k$ and 
$\sum_{i=1}^\infty \|B_i\|<c$.
Then for $E_k:=\bigcup_i B_i$ we have $E_k\subset \bigcup_{x\in A} B(x,1)$, $\|E_k\|<c$
and $|E_k|<\frac ck$.

It follows that $E_k\to\emptyset$ in $\tbv$, whereas $\F(E_k)=\Hnmj(E_k\cap A)=
\Hnmj(A)>0$.
By Proposition \ref{l:charchar} $\F$ cannot be a charge.

It is easy to check that $\F$ is not a charge 
if $\Hnmj(A)=\infty$.

\end{enumerate}
\end{examples}

\section{Packing $\R$ integral}

In this section we set up concept of the packing $\R$ integral, which will be further developed in the next section.

\begin{definition}
A pairwise disjoint finite system of balls 
$(B(x_i,r_i))_{i=1}^k$ 
in $\rn$ is called a \textit{packing}.

A function $\delta:E\to[0,\infty)$, where $E\subset \rn$,  is called a \emph{gage}
if the set $N=\{x;\delta(x)=0\}$ is of $\sigma$-finite $\Hnmj$ Hausdorff measure.

We say that a system $P=\{(A_1,x_1),\ldots, (A_k,x_k)\}$, $A_i\subset \rn$ and $x_i\in \rn$, 
is \emph{$\delta$-fine} if
$d(A_i\cup {x_i})<\delta(x_i)$.
Let us remark that we do not require $x_i\in A_i$.

Especially, a packing 
$(B(x_i,r_i))_{i=1}^k$ is $\delta$-fine if and only if $2r_i<\delta(x_i)$ for $i=1,\ldots, k$.

\end{definition}

\begin{notation}
Let $x\in \er^n$, $r,\varepsilon>0$ and $\F$ be a charge. Then we will use the seminorms
$$
\bar{p}_{x,r}^\varepsilon(\F)=
\sup\left\{
|\F(E)|; E\subset\subset B(x,r), E\in \tbv, (E,x) \mbox{ is }\varepsilon\text{-regular}
\right\}.
$$
\end{notation}

\begin{definition}
Let $A\subset\rn$ be a locally $\BV$ set. 
We say that a charge $\F$ in $A$ is an \textit{indefinite packing $\R$
integral} of a function $f:\cl_*A\to\er$ in $A$ with respect to a charge $\G$
if there exists 
$\tau\in (0,1]$
such that for every $\varepsilon>0$ there exists a gage $\delta\colon \cl_* A\to[0,\infty)$
such that for every $\delta$-fine packing 
$\sys{B(x_i,r_i)}_{i=1}^k$, $x_i\in \cl_* A$, we have
$$
\sum_{i=1}^k
\bar{p}^\varepsilon_{x_i,\tau r_i}(\F-f(x_i)\G)<\varepsilon.
$$
\end{definition}

\begin{remark}
In the previous definition, as well as in forthcoming Definitions
 \ref{d:pih}, \ref{d:dpih}, \ref{d:gdiv}, \ref{def:pff}, \ref{def:pffj} and \ref{def:mp}
it is possible to consider a function $f$ defined only on $\cl_*A\setminus T$,
where $T$ is of $\sigma$-finite $\Hnmj$ Hausdorff measure.
The integral is well defined since we can consider gages $\delta$ with
$\delta=0$ on $T$. For the same reason, the
indefinite packing $\R$ integral with respect to any charge $\G$ 
does not depend on values of $f$ on a set of 
$\sigma$-finite $\Hnmj$ Hausdorff measure.
\end{remark}

\begin{remark}
The uniqueness of the indefinite packing integral of $f$ in $A$ will be discussed later.
\end{remark}

\begin{remark}
The indefinite packing $\R$ integral is linear with respect to a function $f$.
\end{remark}

\section{Packing $\R^*$ integral}

Let us continue with so called packing $\R^*$ integral. We will prove its uniqueness, basic properties and finally we will
formulate and prove the Gauss-Green theorem.
Its definition relies on the concept of  an $\varepsilon$-isoperimetric set, which was introduced by Mal\'{y} and Pfeffer in
\cite{mpf}.
We will be inspired by their work also further in this section.

\begin{definition}
Let $\varepsilon>0$ and $E\subset\rn$ be a bounded $BV$ set. We say that $E$ is \emph{$\varepsilon$-isoperimetric} if for each
$T\in\tbv$
$$
\min\{
P(E\cap T),P(E\setminus T)\}\leq\frac{1}{\varepsilon}P(T,\inn E).
$$

Since $P(T, \inn E)=P(E\cap T,\inn E)$, it is enough to consider only $T\subset E$. (See \cite[Lemma 2.1]{mpf}.)
\end{definition}

\begin{notation}
Let $x\in \er^n$, $r,\varepsilon>0$ and $\F$ be a charge. Then we will use the seminorms
$$
\aligned
\bar{q}_{x,r}^\varepsilon(\F)&=
\sup\{
|\F(E)|; E\subset\subset B(x,r), E\in \tbv, x\in\cl_*E,\\
&  (E,x) \mbox{ is }\varepsilon\text{-regular and }E\text{ is }\varepsilon
\text{-isoperimetric}
\}.
\endaligned
$$

\end{notation}

\begin{definition}\label{d:pih}
Let $A\subset \rn$ be a locally $\BV$ set.
We say that a charge $\F$ in $A$ is an \textit{indefinite packing $\R^*$
integral}
of a function $f\colon \cl_* A\to\er$ in $A$ with respect to a charge $\G$
if there exists 
$\tau\in (0,1]$
such that for every $\varepsilon>0$ there exists a gage $\delta\colon \cl_* A\to[0,\infty)$
such that 
for every $\delta$-fine packing 
$\sys{B(x_i,r_i)}_{i=1}^k$, $x_i\in \cl_* A$,
we have
$$
\sum_{i=1}^k
\bar{q}^\varepsilon_{x_i,\tau r_i}(\F-f(x_i)\G)<\varepsilon.
$$
In the case $A=\rn$ we say that $\F$ is just an \textit{indefinite packing $\R^*$ integral} of $f$ with respect to $\G$.

The family of all functions packing $\R^*$ integrable with respect to a charge $\G$ is denoted by $\P\R^*(\G)$.
\end{definition}

\begin{lemma}\label{l:predkey}
Let $\tau\in(0,1]$ and $\varepsilon>0$.
Then there exists a constant $c_T$ (depending only on $\tau$ and $n$)
 with the following property:
for each function $\Phi:\er\to (0,\infty)$, $x\in \rn$ and $R>0$ there exists
$0<r<R$ such that 
$$
\Phi(10r)+\varepsilon|B(x,10r)|\leq c_T\bigl(\Phi(\tau r)+\varepsilon|B(x,\tau r)|\bigr).
$$
\end{lemma}

\begin{proof}
See \cite[Lemma 3.7]{KM}.
\end{proof}

\begin{lemma}\label{l:pomery}
Let $0<\varepsilon\leq1/(2n)$ and $Q=[0,a_1]\times[0,a_2]\times\cdots\times[0,a_{n}]$
be an $\varepsilon$-regular interval.
Then 
$$
\max\{a_1,\ldots,a_n\}\leq \frac{1}{\varepsilon} \min\{a_1,\ldots,a_n\}.
$$
\end{lemma}

\begin{proof}
For simplicity, let us suppose that $a_1\leq a_2\leq \cdots \leq a_n$.
Since $Q$ is $\varepsilon$-regular, we can estimate
$$
a_n(a_2\cdots a_n)\leq d(Q)\|Q\|<\frac{1}{\varepsilon}|Q|=\frac{1}{\varepsilon}a_1a_2\cdots a_n.
$$
Dividing by $a_2\cdots a_n$ we obtain 
$a_n<\frac{1}{\varepsilon}a_1$, which establishes the formula.
\end{proof}

\begin{lemma}\label{l:regiso}
Let $\varepsilon>0$, $Q$ be an $\varepsilon$-regular interval and $T\in\tbv$, $T\subset Q$ satisfying $|T|\leq |Q|/2$.
Then there exists a constant $\gamma=\gamma(\varepsilon,n)$ such that
\begin{equation}\label{e:regiso}
\Hnmj(\partial Q\cap \partial_*T)\leq
\gamma\Hnmj(\intt Q\cap \partial_*T).
\end{equation}
\end{lemma}
\begin{proof}
At first let $Q$ be a cube. By \cite[Lemma 6.7.2]{Pffin} there exists a constant $\eta=\eta(n)$ such that
\begin{equation}\label{e:cube}
\Hnmj(\partial Q\cap \partial_*T)\leq
\eta\Hnmj(\intt Q\cap \partial_*T).
\end{equation}

Further, let $Q$ be an $\varepsilon$-regular interval.
We can suppose $Q=[0,a_1]\times[0,a_2]\times\cdots\times[0,a_{n}]$, $a_1\leq a_2\leq\cdots\leq a_n$.
Let $L:\rn\to\rn$ be a linear mapping represented by the diagonal matrix
$$
\begin{pmatrix}
a_n/a_1 &0 &\cdots &0\\
0&a_n/a_2 & \cdots &0\\
\vdots&\vdots & \ddots & \vdots \\
0&0 & \cdots &1\\
\end{pmatrix}.
$$
Then $L(Q)$ is a cube and $|L(T)|\leq|L(Q)|/2$.
Moreover, 
$\intt L(Q) \cap \partial_* L(T)=L(\intt Q \cap \partial_* T)$.
Further, we can estimate the Lipschitz constant of $L$ as $\Lip(L)=\max_i\{a_i/a_n\}\leq \frac{1}{\varepsilon}$,
 which follows from Lemma \ref{l:pomery}.
Since $L^{-1}$ can be represented by the matrix
$$
\begin{pmatrix}
a_1/a_n &0 &\cdots &0\\
0&a_2/a_n & \cdots &0\\
\vdots&\vdots & \ddots & \vdots \\
0&0 & \cdots &1\\
\end{pmatrix},
$$
we have  $\Lip(L^{-1})=1$.

Applying Lemma \ref{l:Pflipeo}, inequality \eqref{e:cube} and properties of $L$
we obtain
$$
\aligned
\Hnmj(\partial Q\cap \partial_*T)&=
\Hnmj(L^{-1}(L(\partial Q\cap \partial_* T)))\leq
\Hnmj(\partial L(Q) \cap \partial_* L(T))\\
&\leq
\eta \Hnmj(\intt L(Q) \cap \partial_* L(T))
=
\eta \Hnmj(L(\intt Q \cap \partial_* T))\\
&\leq
\frac{\eta}{\varepsilon^{n-1}} \Hnmj(\intt Q \cap \partial_* T).
\endaligned
$$
Hence \eqref{e:regiso} holds with $\gamma(\varepsilon,n):=\frac{\eta}{\varepsilon^{n-1}}$. 
\end{proof}

\begin{lemma}\label{l:regular}
For every $n\in\en$ there exists an increasing function $\beta:(0,\infty)\to\er$ such that 
every $\varepsilon$-regular interval $Q\subset\rn$  is $\beta(\varepsilon)$-isoperimetric.
\end{lemma}
\begin{proof}
We set $\beta(\varepsilon)=1/(1+\gamma(\varepsilon,n))$, the constant $\gamma(\varepsilon,n)$ being as in Lemma \ref{l:regiso}.
Now let us fix an $\varepsilon$-regular interval $Q$ and a set $T\in\tbv$, $T\subset Q$. We need to show that 
$$
\min\{
P(Q\cap T),P(Q\setminus T)\}\leq\frac{1}{\beta(\varepsilon)}P(T,\inn Q),
$$
Let us assume $|T|\leq |Q|/2$. Since $Q$ is an interval, we have $\intt Q=\intt_*Q$.
Then by Lemma \ref{l:regiso} 
there exists a $\gamma=\gamma(\varepsilon,n)$ such that
$$\aligned
P(T)&\leq \Hnmj(\intt Q\cap\partial_*T)+\Hnmj(\partial Q\cap\partial_*T)
\leq(1+\gamma)\Hnmj(\intt Q\cap\partial_*T)\\
&\leq(1+\gamma)P(T,\inn Q)=\frac{1}{\beta(\varepsilon)}P(T,\inn Q).
\endaligned$$
In the case $|T|>|Q|/2$ we have $|Q\setminus T|<|Q|/2$ and then we obtain
$$
P(Q\setminus T)=P(Q\cap(Q\setminus T))\leq(1+\gamma)P(Q\setminus T,\inn Q)
=\frac{1}{\beta(\varepsilon)}P(T,\inn Q).
$$
\end{proof}

\begin{lemma}\label{l:predkvadry}
Let $r>0$, $x\in\rn$  and $Q=[a_1,b_1]\times[a_2,b_2]\times\cdots\times[a_n,b_n]$  be an 
interval such that 
$Q\subset B(x,2r)$ and  $\frac{r}{2\sqrt{n}}\leq\min_l\{|b_l-a_l|\}$.
Then $(Q,x)$ is $\rho$-regular,
where 
$\rho=\rho(n)=\frac{1}{n^{\frac{n+1}{2}}2^{3n-2}}$.
\end{lemma}
\begin{proof}
Let us denote
$s:=\min_l\{|b_l-a_l|\}$ and $w:=\max_l\{|b_l-a_l|\}$. 
Since $\frac{r}{2\sqrt{n}}\leq s$, $w\leq 4r$ and $\diam(Q\cup\{x\})\leq 4r$,
we can estimate the regularity of $Q$ as
$$
\aligned
r(Q,x)&=
\frac{|Q|}{\diam(Q\cup\{x\})\cdot\|Q\|}\\
&\geq\frac{s^{n-1}w}{4r\cdot 2nw^{n-1}}\\
&\geq\frac{\left(\frac{r}{2\sqrt{n}}\right)^{n-1}}{8rn\left(4r\right)^{n-2}}\\
&=\frac{1}{n^{\frac{n+1}{2}}2^{3n-2}}=\rho(n).\\
\endaligned
$$
\end{proof}

\begin{lemma}\label{l:kvadry}
Let $\F$ be a charge and $B(x,r)\subset\rn$, $x=(x_1,x_2,\ldots,x_n)$, 
be a ball.
Further, let 
$Q=[a_1,b_1]\times[a_2,b_2]\times\cdots\times[a_n,b_n]$  be an 
interval such that
$Q\subset B(x,2r)$ and $\frac{r}{2\sqrt{n}}\leq\min_l\{|b_l-a_l|\}$.
Then 
\begin{equation}\label{e:kvadry}
|\F(Q)|\leq
2^m\bar{q}^{\varepsilon}_{x,2r}(\F),
\end{equation}
where 
$m=\#\{l; x_l\not\in[a_l,b_l]\}$, $\varepsilon=\min\{\beta(\rho),\rho\}$ and $\beta$ and $\rho$ are as in Lemma
\ref{l:regular} and Lemma \ref{l:predkvadry}.
\end{lemma}
\begin{proof}

The proof proceeds by induction on $m$.
First, for $m=0$ we have $x\in Q$.
Since $Q\subset B(x,2r)$ and $\frac{r}{2\sqrt{n}}\leq\min_l\{|b_l-a_l|\}$,
 $Q$ is $\rho(n)$-regular. Furthermore, by Lemma \ref{l:regular} we obtain $Q$ is also $\beta(\rho)$-isoperimetric.
Then we can estimate 
$$
|\F(Q)|\leq
\bar{q}^{\varepsilon}_{x,2r}(\F).
$$

Now let us fix $m\geq 1$ and suppose that \eqref{e:kvadry} holds for $m-1$.
Without loss of generality we can assume that 
$x_l\not\in[a_l,b_l]$ for $l=1,\ldots,m$.

Our next purpose is to define an auxiliary interval 
$$\widetilde{Q}=
[{a}_1,{b}_1]\times\cdots\times
[a_{m-1},b_{m-1}]
\times
[\tilde{a}_{m},\tilde{b}_{m}]\times
[a_{m+1},b_{m+1}]\times
\cdots\times [a_n,b_n],$$
where $[\tilde{a}_{m},\tilde{b}_{m}]$ is defined as follows:

In the case $x_m<a_m$ let us set $\widetilde{a}_m=x_m-(b_m-x_m)$ and  $\widetilde{b}_m=b_m$.
If $x_m>b_m$, let us set $\widetilde{a}_m=a_m$, $\widetilde{b}_m=x_m+(x_m-a_m)$.

We see that  $Q\subset \widetilde{Q}\subset B(x,2r)$ and $x\in\widetilde{Q}$.
For simplicity, let us assume $x_m<a_m$. Then
$$
{Q}=\widetilde{Q} \setminus
\widetilde{Q}', 
$$
where
$$
\widetilde{Q}'=
[{a}_1,b_1]\times\cdots\times
[{a}_{m-1},b_{m-1}]\times
[\tilde{a}_{m},a_{m}]\times
[a_{m+1},b_{m+1}]\times\cdots\times [a_n,b_n]. 
$$

In the following, we need to estimate the regularity of subintervals $\widetilde{Q}$ and $\widetilde{Q'}$.
Since $\min_l\{|b_l-a_l|\}\geq \frac{r}{2\sqrt{n}}$ 
and $\widetilde{Q}\subset B(x,2r)$,
 by Lemma \ref{l:predkvadry} we obtain 
$\widetilde{Q}$ is $\rho(n)$-regular.
Analogously we obtain the regularity of $\widetilde{Q}'$.

By Lemma \ref{l:regular} we have $\widetilde{Q}$ and $\widetilde{Q}'$ are $\beta(\rho)$-isoperimetric.
Using the additivity of $\F$ and the inductive assumption we obtain
$$
|\F(Q)|\leq |\F(\widetilde{Q})|+|\F(\widetilde{Q}')|
\leq 2\cdot 2^{m-1}\bar{q}^{\varepsilon}_{x,2r}(\F)=
 2^{m}\bar{q}^{\varepsilon}_{x,2r}(\F),
$$
which completes the proof.

\end{proof}

\begin{theorem}[Uniqueness of the integral]\label{thm:uniq}
Let $f$ be a function and $\G$ be \linebreak a charge. Then there exists at most one indefinite packing $\R^*$ integral of $f$ with
respect to $\G$.
\end{theorem}

\begin{proof}
Let $\F_1$, $\F_2$ be indefinite packing $\R^*$ integrals of $f$ with respect to $\G$. Then
$\F_1-\F_2$ is the integral of $0$ with respect to $\G$.
So it is sufficient to show that if
 $\F$ is an indefinite packing $\R^*$ integral of $f\equiv 0$, 
then $\F\equiv 0$.

By Lemma \ref{l:aproxkrychle} it is enough to prove that $\F(K)=0$ for each dyadic cube $K$.
Let $\tau$ be as in Definition \ref{d:pih}.
Now, let us fix a dyadic cube $K$ 
of side $a_0$ and choose $\varepsilon>0$ such that
$\varepsilon<\min\{\beta(\rho),\rho\}$,
 where $\beta$ and $\rho$ are as in Lemma \ref{l:regular} and Lemma \ref{l:predkvadry}.
Finally, denote $K_0:=\bigcup_{x\in K}B(x,1)$.

STEP 1.

Since $\F$ is an indefinite packing $\R^*$ integral of $f\equiv0$, there exists a 
gage $\delta:\rn\to [0,\infty)$  
such that for every $\delta$-fine packing 
$\sys{B(x_i,r_i)}_{i=1}^h$
we have
\begin{equation}\label{e:defint}
\sum_{i=1}^h
\bar{q}^\varepsilon_{x_i,\tau r_i}(\F)<\varepsilon.
\end{equation}

STEP 2.

In this step we construct the covering of the set $K\setminus N$, where $N=\{x;\delta(x)=0\}$.

By Lemma \ref{l:predkey}, applied to $\Phi(r):=\bar{q}^\varepsilon_{x,r}(\F)$, we can
find
a constant $c_T$ such that 
for every $x$ there exists $r(x)<\delta(x)$, $10 r(x)<1$, 
with the following properties:
\begin{equation}\label{e:mkrd}
20 r(x)<a_0
\end{equation}
and
\begin{equation}\label{e:predkey}
\bar{q}^\varepsilon_{x,10r(x)} (\F) +\varepsilon|B(x,10r(x))|\leq
c_T (\bar{q}^\varepsilon_{x,\tau r(x)} (\F) +\varepsilon|B(x,\tau r(x))|).
\end{equation}
Now, let us consider the covering $\C=\{\bar{B}(x,r(x)); x\in K\setminus N\}$.
By the Vitali theorem we can construct a pairwise disjoint subsystem 
$\C'\subset \C$, such that
$\bigcup_{{B}(x,R)\in\C''} {B}(x,R)\supset K\setminus N$,
where  
$\C''=\{B(x,5r);\bar{B}(x,r)\in\C' \}$.

STEP 3.

Now we will cover the set $N$.

Since $N$ is of $\sigma$-finite $\Hnmj$ measure, we can write out 
$N=\bigcup_{s=1}^\infty N^s$, where $\H^{n-1}(N^s)=c_s<\infty$ 
for every $s=1,2,\ldots$
Let us fix $s\in\en$
and $\varepsilon_s\in(0,\varepsilon)$ such that 
\begin{equation}\label{e:epsilon}
\varepsilon_s (c_1c_c2^{n-1}(c_s+\varepsilon)+1)<2^{-s}\varepsilon,
\end{equation}
where
$c_1= 2^n n^{(3-n)/2}$ and $c_c=\alpha_n2^{2n}n^{n/2}$.
By Lemma \ref{l:charchar}, with $\varepsilon_s$ we can associate $\theta_s$ such that 
for every $\BV$ set $E\subset B(1/\varepsilon)$ we have
\begin{equation}\label{e:charchar}
|\F(E)|<\theta_s|E|+\varepsilon_s(\|E\|+1).
\end{equation}

Furthermore, there exist $\zeta_s<1/2$ and a system of balls $\N^s=\{B(x_i^s,R^s_i)\}$
covering $N^s$ such that
$R^s_i\leq\zeta_s$, 
\begin{equation}\label{e:mkr}
4R^s_i<a_0,
\end{equation}
$$
c_2\zeta_s\theta_s(c_s+\varepsilon)<2^{-s}\varepsilon\alpha_{n-1}
$$
and
\begin{equation}\label{e:nmna}
\sum_{B(x_i^s,R_i^s)\in\N^s} 
\alpha_{n-1}\left(\frac{\diam B(x_i^s,R^s_i)}{2}\right)^{n-1}\leq
(\alpha_{n-1}+1)\sum_{B(x_i^s,R_i^s)\in\N^s} 
(R^s_i)^{n-1}
<c_s+\varepsilon.
\end{equation}
Note that
\begin{equation}\label{e:nmnadva}
c_2\theta_s \sum_{B(x_i^s,R_i^s)\in\N^s} (R_i^s)^{n}
\leq c_2\zeta_s\theta_s \sum_{B(x_i^s,R_i^s)\in\N^s} (R_i^s)^{n-1}<2^{-s}\varepsilon,
\end{equation}
where $c_2=\alpha_n c_c 2^{n}$.

Let us denote $\N:=\bigcup_s \N^s$.

Now, let us consider the covering $\V:=\C''\cup\N$. Since $\V$ covers the compact set $K$, 
we can choose a finite system of balls
$B(x_i,R_i)\in\V$, $i=1,\ldots, k$, covering $K$.
Without loss of generality we can assume that $B(x_1,R_1),\ldots, B(x_h,R_h)\in\C''$ 
and $B(x_{h+1},R_{h+1}),\ldots,B(x_k,R_k)\in \N$.

STEP 4.

In this step we construct a partition of the cube $K$ in the sense that 
we look for a finite system of nonoverlapping cubes whose union is $K$.

Recall that $Q'$ denotes the mother cube of a cube $Q$.
Let $\K$ denote the family of all dyadic subcubes of $K$.
For fixed $i\in\{1,\ldots,k\}$ set
$$\aligned
\tilde\Q_i=\{Q\in\K;
Q\cap B(x_i,R_i)\neq \emptyset,\, 
Q\subset B(x_{i},2R_{i})
\text{ and } Q'\not\subset B(x_{i},2R_{i})
\}.
\endaligned
$$

We show that the union 
$\tilde\Q=\bigcup^k_{i=1}\tilde\Q_i$
is all of $K$. Choose $y\in K$. 
Consider a sequence $P_l$ of dyadic cubes such that $P_0=K$, 
$P_{l-1}=P_l'$  for $l=1,2,\ldots$ and $\{y\}=\bigcap_{l=0}^\infty 
P_l$. There exists $i\in\{1,\ldots,k\}$ such that $y\in B(x_i,R_i)$.
Since $\diam P_l \searrow 0$, 
there exists $l$ such that  $P_{l}\subset B(x_{i},2R_{i})$.
We find the smallest $l$ such that $P_l\subset B(x_{i},2R_{i})$.
 By \eqref{e:mkrd} and \eqref{e:mkr}, $l\ge 1$.
We easily verify that $y\in P_l\in\tilde\Q$. 

Next we show that the system $\tilde \Q$ is finite.
Let us fix $Q\in\tilde\Q_i$ and 
denote the side length of $Q$ by $a$.
The length of the diagonal can be expressed as $\sqrt{n}a$.
Since
$Q'$ intersects both $B(x_i,R_i)$ and $B(x_i,2R_i)^c$, we obtain
\begin{equation}\label{dolni}
R_i/2<\sqrt{n}a.
\end{equation}
Hence the side length of all cubes in $\tilde\Q_i$ is bounded from below.
Therefore, 
the systems $\tilde\Q_i$ and hence the system $\tilde\Q$ are finite.

Now we can define 
the system of cubes 
$$
\Q=\tilde\Q\setminus\{
Q\in\tilde\Q;\, \exists P\in\tilde\Q\text{ such that }\, P\supsetneq Q
\}.
$$
Since two dyadic cubes are either in inclusion or nonoverlapping,
$\Q$ is a finite partition of $K$; we enumerate it as
$\Q=\{Q_j, j=1,\ldots, m\}$.
Finally, let us define the systems
$$
\Q_i=\Big\{Q\in \Q\cap\tilde \Q_i\colon Q\notin \bigcup_{l<i} \tilde \Q_l\Big\}.
$$

Let us fix $Q_j\in \Q_i$ and denote it side length by $a_j$.
Recall that the length of the diagonal can be expressed as $\sqrt{n}a_j$ and 
since $Q_j$ is included in $B(x_i,2R_i)$, we have
 $\sqrt{n}a_j<4R_{i}$. Hence
we can estimate the perimeter of $Q_j$:
\begin{equation}\label{e:perQ}
\|Q_j\|=2n a_j^{n-1}
\leq 2n\left(\frac{4R_i}{\sqrt{n}}\right)^{n-1}
= 2^nn^{(3-n)/2}2^{n-1}R_i^{n-1}=
c_12^{n-1}R_i^{n-1}.
\end{equation}
Let us estimate the number of the cubes $Q_j\in\Q_i$. 
Applying \eqref{dolni}, we obtain 
$$
\alpha_n(2R_i)^n=|B(x_i,2R_i)|\geq \left| \bigcup_{Q_j\in\Q_i} Q_j\right|\geq \#\Q_i \left(\frac{R_i}{2\sqrt{n}}\right)^n.
$$
Hence $\# \Q_j\leq c_c$. 
(Let us remind that $c_c= \alpha_n 2^{2n}n^{n/2}$.)

STEP 5. 

In this step we estimate $\F(K)$. 
By the additivity of $\F$ we obtain

$$
\F(K)=
\F\left(\bigcup_{i=1}^k \bigcup_{Q_j\in \Q_i} Q_j \right)=
\F\left(\bigcup_{i=1}^h \bigcup_{Q_j\in \Q_i} Q_j \right)+
\F\left(\bigcup_{i=h+1}^k \bigcup_{Q_j\in \Q_i} Q_j \right).
$$

Firstly let us suppose that $i\in\{1,\ldots, h\}$.
Then $B(x_{i},R_{i})\in\C''$.
Let us fix a  pair $(Q_j,B(x_i,R_i))$. Since $q_j\geq\frac{R_i}{2\sqrt{n}}$, 
we can apply Lemma \ref{l:kvadry} and obtain
\begin{equation}\label{e:konkrkvadr}
|\F(Q_j)|\leq 2^n\bar{q}^\varepsilon_{x_i,2R_i}(\F).
\end{equation}

Using the fact that $\#Q_j\leq c_c$, the system $\{B(x_1,r_1),\ldots, B(x_h,r_h)\}$ is a $\delta$-fine packing 
 and applying \eqref{e:konkrkvadr}, \eqref{e:predkey} and \eqref{e:defint} we can estimate
$$
\aligned
\left|\F\left(\bigcup_{i=1}^h \bigcup_{Q_j\in \Q_i} Q_j \right)\right|&=
\sum_{i=1}^h \sum_{Q_j\in \Q_i} |\F(Q_j)|\\
&\leq
\sum_{i=1}^h c_c2^n\bar{q}^\varepsilon_{x_i,2R_i}(\F)\\
&\leq 2^nc_c\sum_{i=1}^h 
c_T(\bar{q}^\varepsilon_{x_i,\tau r_i}(\F)+\varepsilon|B(x_i,\tau r_i)|)\\
&<c_c2^n(c_T\varepsilon +c_T\varepsilon|K_0|)
=c_c2^nc_T\varepsilon (1+|K_0|).
\endaligned
$$

Secondly, let us fix $s\in\en$ and set 
$\A^s:=\{i\in\{h+1,\ldots,k\}; B(x_{i},R_{i})\in\N^s\}$.
Then, applying the fact that $\# \Q_j\leq c_c$ and inequalities \eqref{e:charchar}, \eqref{e:nmna},
 \eqref{e:epsilon}, \eqref{e:perQ} and \eqref{e:nmnadva} we obtain
$$
\aligned
\left|\F\left(\bigcup_{{i\in\A^s }} \bigcup_{Q_j\in\Q_i}
Q_j\right)\right|
&\leq
\theta_s\left|\bigcup_{{i\in\A^s }} \bigcup_{Q_j\in\Q_i}
Q_j\right|
+\varepsilon_s\left(\left\|\bigcup_{{i\in\A^s }} \bigcup_{Q_j\in\Q_i}
Q_j
\right\|+1\right)\\
&<
\varepsilon_s + \theta_s \alpha_n c_c2^{n}\sum_{{i\in\A^s }} 
 R_i^n +\varepsilon_s
\sum_{{i\in\A^s }} 
\sum_{{Q_j\in\Q_i }} 
\|Q_j\| \\
&<
\varepsilon_s + \theta_s  c_2\sum_{{i\in\A^s }} 
 R_i^n +\varepsilon_s c_c c_1
2^{n-1}\sum_{{i\in\A^s }} R_i^{n-1}\\
&\leq
\varepsilon_s + 
\varepsilon 2^{-s} +\varepsilon_s c_c c_1 2^{n-1}(c_s+\varepsilon)
<
2\cdot2^{-s}\varepsilon.
\endaligned
$$

Since the union 
$
\bigcup_{s=1}^\infty\bigcup_{{i\in\A^s }} 
\bigcup_{Q_j\in\Q_i} Q_j
$
has only finite number of nonempty elements, we can use the additivity of $\F$ and  we obtain

$$
|\F(K)|<
c_c2^nc_T\varepsilon(1+|K_0|)+\varepsilon\sum_{s=1}^\infty 2^{-s+1}
=
\varepsilon(c_Tc_c2^n(1+|K_0|)+2),
$$
which completes the proof.

\end{proof}

\begin{remark}
The indefinite packing $\R^*$ integral of  a function $f$ with respect to a charge $\G$ depends linearly on $f$.
\end{remark}

In the preceding, we were concerned with an indefinite packing $\R^*$ integral of a function $f:\rn\to\er$.
Now we will concentrate on a packing $\R^*$ integral in $A$, where $A$ is a locally $\BV$ set.

\begin{theorem}\label{thm:rhr}
Let $A\subset\rn$ be a locally $\BV$ set and let a 
charge $\F$ be an indefinite packing $\R$ integral of a function $f:\cl_*A\to\er$ in $A$
 with respect to a charge $\G$.
Then $\F$ is also an indefinite packing $\R^*$ integral of $f$ in $A$ with respect to $\G$.
\end{theorem}
\begin{proof}
The proof follows from the fact that $\bar{q}^\varepsilon_{x,r}\leq \bar{p}^\varepsilon_{x,r}$.
\end{proof}

\begin{notation}
Let $A\subset\rn$ and $f:A\to\er$ be a function. Then $\bar{f}_A$ denotes the
zero extension of $f$:
$$
\bar{f}_A= \begin{cases}
f(x)  & \text{if } x\in A,\\
0  & \text{if } x\not\in A.\\
\end{cases}
$$
\end{notation}

The two following lemmas with proofs can be found in \cite[Lemma 2.5 and 3.7]{mpf}.
\begin{lemma}\label{l:meas}
Let $\F$ be a charge. Then for $\varepsilon>0$ there is an absolutely continuous Radon measure $\mu$ in $\rn$ such that
for each $BV$ set $E\subset B(0,1/\varepsilon)$,
$$
|\F(E)|\leq \mu(E)+\varepsilon P(E).
$$
\end{lemma}

\begin{lemma}\label{l:admis}
Let $A\in\cBV_{\loc}$ and $\varepsilon>0$. For each $x\in \ext_c A$, there is $\delta>0$ such that
every strongly $\varepsilon$-regular set $E$ with $x\in\cl_* E$ and $d(E)<\delta$ satisfies
$$
P(E\cap A)\leq P(E\setminus A).
$$
\end{lemma}

The proof of the next theorem follows the lines of the proof in \cite[Lemma 3.8]{mpf}.

\begin{theorem}\label{thm:admis}
Let $\F$ be a charge and $A$ be an admissible locally $\BV$ set.
For given $\tau\in(0,1]$ and $\varepsilon>0$ there is a 
gage $\delta:\rn\to [0,\infty)$ such that 
$$
\sum_{x_i\in A}\bar{q}^\varepsilon_{x_i,\tau r_i}(\F\lfloor_{A^c})<\varepsilon 
\quad \mbox{and }\quad 
\sum_{x_i\not \in A}\bar{q}^\varepsilon_{x_i,\tau r_i}(\F\lfloor_A)<\varepsilon
$$
for each $\delta$-fine packing $\sys{B(x_i,r_i)}_{i=1}^k$.
\end{theorem}
\begin{proof}
At first let us suppose that $A$ is bounded.
Let us fix $\varepsilon>0$ such that $\bar A\subset B:=B(0,1/\varepsilon')$,
where 
\eqn{ep'}
$$\ep'=\frac{\ep^2}{P(A)}.
$$
By Lemma \ref{l:meas}, there is an absolutely continuous Radon measure $\mu$ in $\rn$ such that 
$$
|F(E)|\leq \mu(E)+\ep'\,P(E)
$$
for each $E\in\BV$, $E\subset B$.
Then there exists a compact $K$ such that  $K\subset B\setminus A$ and 
\begin{equation}\label{e:mi}
\mu((B\setminus A)\setminus K)<\frac12\,\varepsilon.
\end{equation}
Applying Lemma \ref{l:admis} to $A^c$, for each $x\in B\cap \ext_c A^c=B\cap \intt_c A$ we can find $\delta_x>0$ such that $B(x,\delta_x)\subset B$,
and
\begin{equation}\label{e:p}
P(E\setminus A)\leq P(E\cap A)
\end{equation}
for each strongly $\varepsilon$-regular set $E$ with $x\in\cl_*E$ and $d(E)<\delta_x$.

Making $\delta_x$ smaller, we may assume that $K\cap B(x,\delta_x)=\emptyset$ for $x\in  \intt_c A$.
Since $A$ and  is an admissible set, it follows that also $A^c$ is admissible and hence  $\intt_c A^c\subset A^c$ and $A^c\cap \ext_cA^c=\emptyset$.
Let us set $N:=\partial_c A^c=\partial_c A$. Then $N$ is of $\sigma$-finite Hausdorff measure $\Hnmj$, which follows from the criterion for
finite perimeter \cite[p. 222]{EG}.
Now we can define a gage $\tilde\delta$ on $\rn$ in the following way:
$$
\tilde\delta(x) = \begin{cases}
0  & \text{if } x\in N,\\
1  & \text{if } x\in \ext_c A,\\
\delta_x & \text{if } x\in \intt_c A.
\end{cases}
$$

Let us fix a $\tilde\delta$-fine packing $\sys{B(x_i,r_i)}_{i=1}^k$ 
and sets $E_i$, where 
$E_i\subset\subset B(x_i,\tau r_i)$, $E_i\in \BV$, $(E_i,x_i)$ is  $\varepsilon$-regular and $E_i$ is  $\varepsilon$-isoperimetric 
for each $i=1,\ldots, k$.
By the $\varepsilon$-regularity of $E_i$, inequality \eqref{e:p} and definition of $\tilde\delta$, we obtain
\begin{enumerate}
\item $x_i\not\in N$ for $i=1,\ldots, k$;
\item $E_i\setminus A\subset (B\cap A^c)\setminus K$ when $x_i\in \intt_cA$;
\item $P(E_i\setminus A)\leq (1/\varepsilon)P(A^c,\inn E_i)$ when $x_i\in \intt_c A$.
\end{enumerate}

Hence, using the inequality \eqref{e:mi} and 
the fact that packing is pairwise disjoint, we can estimate 
$$
\aligned
\sum_{x_i\in A} |\F(E_i\setminus A)|&=
\sum_{x_i\in \intt_c A} |\F(E_i\setminus A)|
\\
&\leq
\sum_{x_i\in \intt_c A} \mu(E_i\setminus A)+\frac{\varepsilon^2}{2P(A)}\;P(E_i\setminus A)\\
&\leq
\mu(B\cap A^c\setminus K)+\frac{\varepsilon}{2P(A)}
\sum_{x_i\in \intt_c A} P(A^c, \inn E_i)\leq
\frac\varepsilon2+\frac\varepsilon2.
\endaligned
$$
Passing to the supremum we obtain
$$
\sum_{x_i\in A}\bar{q}^\varepsilon_{x_i,\tau r_i}(\F\lfloor_{A^c})<\varepsilon, 
$$
which we needed.

We now turn to the case $A$ is unbounded. 
Let us consider a  sequence of balls $\{B_m\}$ which forms a 
locally finite
covering of $\rn$. 
Choose $\ep>0$.
Let us fix $m\in\en$ and set $A_m=A\cap B_m$. Then $A_m$ is a bounded admissible locally $\BV$ set
and we can use the previous step to find $\ep_m\le 2^{-m}\ep$ and 
a gage $\delta_m:\rn\to[0,\infty)$ such that 
$$
\sum_{x_i\in A_m}\bar{q}^\varepsilon_{x_i,\tau r_i}(\F\lfloor_{A_m^c})<
\varepsilon_m
$$
for every $\delta_m$-fine packing $((B(x_i,r_i))_{i=1}^k$.

Further, let us set 
$$
\tilde\delta(x):=\min\{\delta_m(x)\colon x\in B_m\}.
$$ 
It is easily seen that $\tilde\delta$ is a gage.
Let us fix a $\tilde\delta$-fine packing $((B(x_i,r_i))_{i=1}^k$. Then
$$
\aligned
\sum_{x_i\in A}\bar{q}^\varepsilon_{x_i,\tau r_i}(\F\lfloor_{A^c})&\leq
\sum_{m=1}^\infty 
\sum_{x_i\in A_m}\bar{q}^\varepsilon_{x_i,\tau r_i}(\F\lfloor_{A_m^c})\\
&<
\sum_{m=1}^\infty 2^{-m}\varepsilon=\ep,
\endaligned
$$
which establishes the formula.

Finally, we proceed similarly to find 
$\tilde\delta^c$ which yields the second inequality and set
$$
\delta=\min\{\tilde\delta,\tilde\delta^c\},
$$
which gives both inequalities at the same time.
\end{proof}

In the proof of the next theorem we are inspired by \cite[Proposition 3.9]{mpf}.
\begin{theorem}\label{thm:urc}
Let $\G$, $\F$ be charges, $f\in\P\R^*(\G)$ and let $\F$ be an indefinite packing $\R^*$ integral of $f$ with respect to $\G$.
If $A$ is an admissible locally $\BV$ set, then $\chi_Af\in\P\R^*(\G)$ and $\F\lfloor_A$ is an indefinite packing $\R^*$ integral
of $\chi_A f$ with respect to $\G$.
\end{theorem}
\begin{proof}
Let us fix $\tau\in(0,1]$ as in Definition \ref{d:pih} and $\varepsilon>0$. 
By the definition of packing $\R^*$ integral and Theorem \ref{thm:admis}
 there exists  a gage $\delta:\rn\to[0,\infty)$ such that
for every $\delta$-fine packing $(B(x_i,r_i))_{i=1}^k$ 
we have
$$
\sum_{i=1}^k \bar{q}^\varepsilon_{x_i,\tau r_i} (\F-f(x_i)\G)<\varepsilon,
$$
$$
\sum_{x_i\in A}\bar{q}^\varepsilon_{x_i,\tau r_i}(\F\lfloor_{A^c})<\varepsilon 
\quad \mbox{and }\quad 
\sum_{x_i\not \in A}\bar{q}^\varepsilon_{x_i,\tau r_i}(\F\lfloor_A)<\varepsilon. 
$$

Hence
$$
\aligned
&\sum_{i=1}^k \bar{q}^\varepsilon_{x_i,\tau r_i} (\F\lfloor_A-f(x_i)\chi_A(x_i)\G)\\
&\qquad=
\sum_{x_i\in A} \bar{q}^\varepsilon_{x_i,\tau r_i} (\F\lfloor_A-f(x_i)\G)+
\sum_{x_i\not\in A} \bar{q}^\varepsilon_{x_i,\tau r_i} (\F\lfloor_A)\\
&\qquad<
\sum_{x_i\in A} \bar{q}^\varepsilon_{x_i,\tau r_i} (\F\lfloor_A-\F)+
\sum_{x_i\in A} \bar{q}^\varepsilon_{x_i,\tau r_i} (\F-f(x_i)\G)+
\sum_{x_i\not\in A} \bar{q}^\varepsilon_{x_i,\tau r_i} (\F\lfloor_A)\\
&\qquad=
\sum_{x_i\in A} \bar{q}^\varepsilon_{x_i,\tau r_i} (\F\lfloor_{A^c})+
\sum_{x_i\in A} \bar{q}^\varepsilon_{x_i,\tau r_i} (\F-f(x_i)\G)+
\sum_{x_i\not\in A} \bar{q}^\varepsilon_{x_i,\tau r_i} (\F\lfloor_A)\\
&\qquad<3\varepsilon,
\endaligned
$$
which completes the proof.
\end{proof}

\begin{theorem}\label{thm:admisn}
Let $A$ be an admissible locally $\BV$ set.
Then the charge $\F$ in $A$ is an indefinite packing $\R^*$
integral of a function $f:\cl_*A\to\er$ with respect to a charge $\G$ in $A$ if and only if
$\F$ is an indefinite packing $\R^*$ integral of $\bar{f}_A$ with respect to $\G$ in $\rn$.
\end{theorem}

\begin{proof}
Let us suppose that $\F$ in $A$ is the indefinite packing $\R^*$ integral of $f$ with respect to $\G$ in $A$. 
Let us fix $\varepsilon>0$.
Now let $\tau\in (0,1]$ and a gage $\delta_1$ on $\cl_*A$ be as in Definition \ref{d:pih} 
and let $\delta_2$ on $\rn$ be as in Theorem \ref{thm:admis}. 
Then let us fix a $\delta$-fine packing $\sys{B(x_i,r_i)}_{i=1}^k$ and 
set $$\delta=
\begin{cases}
\min\{\delta_1(x),\delta_2(x)\} & \text{if } x\in \cl_*A,\\
\delta_2(x) & \text{if } x\in \rn\setminus \cl_*A.
\end{cases}
$$

At first, let us consider the sum over $x_i\in A$.
Since $\F$ in $A$ is the indefinite packing $\R^*$ integral of $f$ in $A$, we have
$$
\sum_{x_i\in A}
\bar{q}^\varepsilon_{x_i,\tau r_i}(\F-\bar{f}_A(x_i)\G)=
\sum_{x_i\in A}
\bar{q}^\varepsilon_{x_i,\tau r_i}(\F-f(x_i)\G)<\varepsilon.
$$

Further, for the case $x_i\not\in A$,
we have by Theorem \ref{thm:admis} the estimate
$$
\sum_{x_i\not\in A}
\bar{q}^\varepsilon_{x_i,\tau r_i}(\F-\bar{f}_A(x_i)\G)=
\sum_{x_i\not\in A}
\bar{q}^\varepsilon_{x_i,\tau r_i}(\F)<\varepsilon.
$$

Therefore we obtain
$$
\aligned
\sum_{i=1}^k
\bar{q}^\varepsilon_{x_i,\tau r_i}(\F-\bar{f}_A(x_i)\G)&=
\sum_{x_i\in A}
\bar{q}^\varepsilon_{x_i,\tau r_i}(\F-\bar{f}_A(x_i)\G)+
\sum_{x_i\not\in A}
\bar{q}^\varepsilon_{x_i,\tau r_i}(\F-\bar{f}_A(x_i)\G)\\
&<
2\varepsilon.
\endaligned
$$
Hence $\F$ is the indefinite packing $\R^*$ integral of $\bar{f}_A$ with respect to $\G$.

Conversely, let $\F$ be the indefinite packing $\R^*$ integral of $\bar{f}_A$ with respect to $\G$.
By Theorem \ref{thm:urc} it follows that $\F\lfloor_A$ is the indefinite packing $\R^*$ integral of $\bar{f}_A$ with respect to
$\G$ in $\rn$. 
In other words, for fixed $\varepsilon>0$ 
there exists  a gage $\delta:\rn\to[0,\infty)$ such that
$\delta=0$ on $\cl_*A\setminus A$ and 
for every $\delta$-fine packing $(B(x_i,r_i))_{i=1}^k$ 
we have
$$
\sum_{i=1}^k \bar{q}^\varepsilon_{x_i,\tau r_i} (\F\lfloor_A-\bar{f}_A(x_i)\G)<\varepsilon.
$$
By the uniqueness of packing $\R^*$ integral we have $\F\lfloor_A=\F$ and hence
$$
\sum_{x_i\in \cl_*A} \bar{q}^\varepsilon_{x_i,\tau r_i} (\F-f(x_i)\G)
=
\sum_{x_i\in A} \bar{q}^\varepsilon_{x_i,\tau r_i} (\F-\bar{f}_A(x_i)\G)
<\varepsilon,
$$
which we needed.

\end{proof}

\begin{corollary}
Let $A$ be an admissible locally  $BV$ set and let a charge $\F$ be an indefinite packing $\R^*$ integral of a 
function $f:\cl_*A\to\er$ in $A$ 
with respect to a charge $\G$. 
Then, by Theorem \ref{thm:admisn}, $\F$ is the indefinite packing $\R^*$ integral of $\bar{f}_A$ with respect to $\G$,
which is unique by Theorem \ref{thm:uniq}.
Therefore the indefinite packing $\R^*$ integral in $A$ is unique as well.

Further, let $A$ be an admissible locally $BV$ set 
and let a charge $\F$ be an indefinite packing $\R$ integral of a function $f:\cl_*A\to\er$ in $A$ 
with respect to a charge $\G$. 
Then $\F$ is also the packing $\R^*$ integral of in $A$ with respect to $\G$ by Theorem \ref{thm:rhr}.
Hence the uniqueness holds also for the indefinite packing $\R$ integral in $A$.
\end{corollary}

\begin{remark}\label{r:wrtL}
Since the function $f$ is defined on $\cl_* A$, the requirement that $A$ be admissible might seem to be unnecessary.
This is really the case with $\G=\L$, because 
sets of measure zero (such as $A\triangle \cl_*A$) does not play a role in integration with respect to Lebesgue measure.
On the other hand, Lebesgue null sets cannot be neglected in general.
For example, the classical Cantor set cannot be neglected
for integration with respect to the Cantor measure in $\er$, which is a charge
by Example \ref{r:ident}(1).
\end{remark}

\begin{definition}\label{d:dpih}
Let $A\in\tbv$ be an admissible set, $f:\cl_*A\to\er$ be a function and $\F$, $\G$ be charges.
We say that the number $\F(A)$ is a \textit{definite packing $\R^*$
integral} of $f$ over $A$ with respect to $\G$
if $\F$ is an indefinite packing $\R^*$ integral of $\bar{f}_A$ with respect to $\G$.

More generally: if $A\subset \rn$ is a bounded measurable  
set and $\F$ is the indefinite packing $\R^*$ integral of $\bar{f}_A$ with respect to $\G$,
then the definite packing $\R^*$ integral of $f$ over $A$ with respect to $\G$ is the number
$\F(A')$, where $A'\in\tbv$, $A'\supset A$ is a bounded admissible set. 

The family of all functions packing $\R^*$ integrable with respect to $\G$ over $A$ is denoted by $\P\R^*(A,\G)$.
\end{definition}

\begin{remark}
The integral does not depend on the choice of $A'$.
Indeed, let $A'$ and $A''$ be bounded admissible $\BV$ sets. 
Since $\bar{f}_{A}\cdot\chi_{A'}=\bar{f}_{A}\cdot\chi_{A''}$,
 by Theorem \ref{thm:urc} and by the uniqueness of the packing $\R^*$
integral we obtain $\F\lfloor_{A'}=\F\lfloor_{A''}$.
Then 
$
\F(A')=\F\lfloor_{A'}(A'\cup A'')=\F\lfloor_{A''}(A'\cup A'')=\F(A''). 
$
\end{remark}

\begin{remark}
Let $A\in \tbv$ be an admissible set, $\G$ be a charge and $f\in\P\R^*(A,\G)$.
Let $\F$ be the  indefinite packing $\R^*$ integral of $f$
 in $A$ with respect to $\G$.
Then the definite $\P\R^*$ integral of $f$ over $A$ wich respect to $\G$ is just
$\F(A)$.
This fact follows from Theorem \ref{thm:admisn}.
\end{remark}

\begin{remark}
If $f$ is a merely an indefinite packing $\R^*$ integrable function,
it does not make sense to define the definite integral over unbounded sets in general.
If we want to set up the definite integral over an unbounded set, we must suppose some additional limiting behaviour
of the indefinite integral at infinity.
There are several nonequivalent ways how to do it and we do not pursue this direction.
\end{remark}

\begin{remark}
Let $A\subset\rn$ be a bounded measurable set and let $f:A\to\er$ be a Lebesgue integrable function.
Then $\bar{f}_A$ is also a Lebesgue integrable function.

Then there exists an indefinite packing $\R^*$ integral of $\bar{f}_A$ with respect to Lebesgue measure.
Hence the definite packing $\R^*$ integral 
of $f$ over $A$
is well defined.
\end{remark}

In the following theorem, we will focus on the convergence of a sequence of sets. 
The importance of this property will be demonstrated in Section \ref{s:sedm}.
The proof uses ideas from Pfeffer and Mal\'y in \cite[Theorem 3.20]{mpf}.

\begin{theorem}\label{thm:conmn}
Let $A$ be a bounded admissible $\BV$ set, $\G$ and $\F$ be charges and let $f:\rn\to\er$ be a function.
Let $\{A_j\}_{j=1}^\infty$ be a sequence of bounded admissible $BV$ sets such that
 $A_j\subset A$ for $j=1,2,\ldots$ and $A_j\to A$ in $\tbv$.
Further, let $f\chi_{A_j}\in \P\R^*$ and 
$\F\lfloor_{A_j}$ be an indefinite packing $\R^*$ integral of $f\chi_{A_j}$ with respect to $\G$ 
with constants $\tau_j$ as in Definition \ref{d:pih}. 
Let $\inf_j \tau_j>0$. 
Then there exists an indefinite packing $\R^*$ integral of $f\chi_{A}$ with respect to $\G$ and is equal to $\F\lfloor_{A}$.
\end{theorem}

\begin{proof}

Let us fix $\tau=\inf_j \tau_j$ and let us denote $N:=A\setminus \bigcup_{j=1}^\infty A_j$.
Then
$N$ is of $\sigma$-finite $\Hnmj$ measure (see \cite[Cor. 6.2.7]{Pf}).
Let us choose $\varepsilon>0$.
Since $\bar{q}^\varepsilon_{x,\tau r}\leq
\bar{q}^\varepsilon_{x,\tau' r}$ for $\tau\leq \tau'$, 
we can by the definition of packing $\R^*$ integral  and by Theorem \ref{thm:admis}
for $j\in\en$ find a gage $\delta_j$ such that
for each $\delta_j$-fine packing $\sys{B(x_i,r_i)}_{i=1}^k$ 
we obtain
\begin{equation}\label{e:conmnj}
\sum_{x_i\in A_j}
\bar{q}^{\varepsilon}_{x_i,\tau r_i}(\F\lfloor_{A_j}-f(x_i)\G)<\varepsilon 2^{-j}
\end{equation}
and
\begin{equation}\label{e:conmnd}
\sum_{x_i\in A_j}
\bar{q}^\varepsilon_{x_i, \tau r_i}(\F\lfloor_{A_j^c})
<\varepsilon 2^{-j}.
\end{equation}

Further, for $x\in\bigcup_{j=1}^\infty A_j$ let us set $j_x:=\min\{j\in\en; x\in A_j\}$.
Now we can define a gage
$$
\delta(x) = \begin{cases}
\delta_{j_x}(x)  & \text{if } x\in\bigcup_{j=1}^\infty A_j,\\
0& \text{if } x\in N.\\
\end{cases}
$$

By Theorem \ref{thm:admisn} it is enough to show that $\F\lfloor_A$ is the indefinite packing $\R^*$ integral of $f$ in $A$
with respect to $\G$.
Let us choose $\delta$-fine packing $(B(x_i,r_i))_{i=1}^k$, $x_i\in A$, and denote $j_i:=j_{x_i}$.
Using the additivity of $\F$ and estimates \eqref{e:conmnj} and \eqref{e:conmnd} we can for fixed $p\in\en$ estimate
$$
\aligned
\sum_{x_i:j_i=p}
\bar{q}^{\varepsilon}_{x_i,\tau r_i}(\F\lfloor_{A}-f(x_i)\G)
&\leq 
\sum_{x_i:j_i=p}
\bar{q}^{\varepsilon}_{x_i,\tau r_i}(\F\lfloor_{A_j}-f(x_i)\G)
+
\bar{q}^\varepsilon_{x_i, \tau r_i}(\F\lfloor_{A\setminus A_j})\\
&<
\varepsilon 2^{-p+1}.\\
\endaligned
$$
Summing over $p$ we obtain
$$
\aligned
\sum_{p=1}^\infty
\sum_{x_i:j_i=p}
\bar{q}^{\varepsilon}_{x_i,\tau r_i}(\F\lfloor_{A}-f(x_i)\G)
&<\sum_{p=1}^\infty
\varepsilon 2^{-p+1}
=2\varepsilon,
\endaligned
$$
which completes the proof.

\end{proof}

\begin{definition}\label{d:gdiv}
Let $A\subset \rn$ be a locally $\BV$ set and let $f:\cl_*A\to\er$ and $\mathbf u\in C(\bar{A},\rn)$ be functions.
Further, let a charge $\F$ be the flux  of $\mathbf u$ in $A$.
We say that $f$ is a \emph{generalized divergence} of $\mathbf u$ in $A$ 
if $\F$ is an indefinite packing $\R^*$ integral of $f$ in $A$.
The generalized divergence of $\mathbf u$ will be denoted by $\Div \mathbf u$.
\end{definition}

The following three definitions was mentioned by Pfeffer in Chapters 2.3 and 2.5 of \cite{Pf}.

\begin{definition}
Let $A\subset\rm$ be a measurable set and let $x\in A\cap \intt_* A$.
 A map $\mathbf u:A\to\rn$ is called \emph{differentiable at $x$ relative to $A$}
if there is a linear map $L:\rm\to\rn$ such that for given $\varepsilon>0$ there exists a $\delta>0$ so that
$$
|\mathbf u(y)-\mathbf u(x)-L(y-x)|<\varepsilon |y-x|
$$
for each $y\in A\cap B(x,\delta)$.
The linear map $L$ is called the 
\emph{differential of $\mathbf u$ at $x$ relative to $A$} and is denoted by $D_A\mathbf u(x)$.

Let $x\in\intt_* A$ and $\mathbf u:\cl_*A\to\er^m$ be a vector field. Let $\mathbf u$ be differentiable at $x$ relative to
$\cl_* A$.
The \emph{divergence of $\mathbf u$ at $x$ relative to $\cl_* A$} is the number
$\div_* \mathbf u(x):=\tr D_{\cl_*A}\mathbf u(x)$,
where $\tr D_{\cl_*A}\mathbf u(x)$ denotes the trace of the matrix representation of the linear transformation 
$D_{\cl_*A}\mathbf u(x):\rm\to\rm$.

By $\div \mathbf u$ we will denote the pointwise divergence defined on interior points of $A$ at
which $\mathbf u$ is differentiable.
Especially, $\div_*\mathbf u=\div \mathbf u$ whenever $\div \mathbf u$ is defined.
\end{definition}

\begin{definition}
Let $\F$ be a charge and let $x\in\rn$. Then for $\eta\geq0$ we define
$$
\underline{D}_\eta\F(x):=\sup_{\delta>0}\inf_E \frac{\F(E)}{|E|}
\quad\mbox{ and }\quad
\overline{D}_\eta\F(x):=\inf_{\delta>0}\sup_E \frac{\F(E)}{|E|},
$$
where $E\in\BV$ such that $d(E\cup\{x\})<\delta$ and $r(E,x)>\eta$.

The \emph{lower} and \emph{upper derivative} of $\F$ at $x$ are defined as
$$
\underline{D}\F(x):=\inf_{\eta>0}\underline{D}_\eta\F(x)
\quad\mbox{ and }\quad
\overline{D}\F(x):=\sup_{\eta>0}\overline{D}_\eta\F(x).
$$

We say that $\F$ is \emph{derivable} at $x$, if 
$$
\underline{D}\F(x)=
\overline{D}\F(x)\neq \pm\infty.
$$
The \emph{derivative} of $\F$ at $x$ is then defined as $D\F(x):=\underline{D}\F(x)=\overline{D}\F(x)$.
\end{definition}

\begin{definition}
Let $E$ be a locally $\BV$ set, $\mathbf u:\cl_*E\to\rn$ be a bounded Borel measurable vector field 
and $\F$ be the flux of $\mathbf u$.
If $\F$ is derivable at $x\in\intt_c E$, we call the number
$\ndiv \mathbf u(x):=D\F(x)$
the \emph{mean divergence} of $\mathbf u$ at $x$.
\end{definition}

Applying the inclusion between $\R$ integral and packing $\R^*$ integral we can state 
sufficient conditions for existence of generalized divergence. For further details see
\cite[Example 2.3.2, Remark 2.5.9, Theorem 5.1.12, Proposition 2.5.7 and Corollary 5.1.13]{Pf}.

\begin{proposition}
Let $A$ be a locally $BV$ set and $\mathbf u\in C(\bar{A},\rn)$.
\begin{enumerate}
\item
If $A=\rn$ and $\mathbf u$ is differentiable in $\rn$, 
then $\div \mathbf u$ is a generalized divergence of $\mathbf u$.

\item
If $\mathbf u$ is differentiable  relatively to $\cl_* A$ on $\intt_c A$,
then $\div_* \mathbf u$ is a generalized divergence of $\mathbf u$.

\item
If $\mathbf u$ is differentiable  relatively to $\cl_* A$ on $\intt_c A$,
then $\ndiv \mathbf u$ is a generalized divergence of $\mathbf u$.

\item
If $\mathbf u$ is Lipschitz on $\cl_* A\setminus T$, where $T$ is of $\sigma$-finite Hausdorff measure $\Hnmj$,
then $\div_* \mathbf u$ is a generalized divergence of $\mathbf u$.

\end{enumerate}
\end{proposition}

\begin{theorem}[Gauss-Green divergence theorem]
Let $A\subset\rn$ be a bounded $BV$ set,
let $\mathbf u\in C(\bar{A},\rn)$. Let us suppose that there exists a generalized divergence $\Div \mathbf u$ 
in $A$.
 Then
$$
\int_{A}\Div \mathbf u(x)\,dx= \int_{\partial_*A}\mathbf u\cdot\boldsymbol\nu_A\,d\Hnmj,
$$
where the integral on the left is the definite packing $\R^*$ integral.
\end{theorem}
\begin{proof}
Since 
$|A\triangle \cl_*A|=0$,
it is enough to show that
$\int_{\cl_* A}\Div \mathbf u(x)\,dx= \int_{\partial_*A}\mathbf u\cdot\boldsymbol\nu_A\,d\Hnmj$ (see Remark \ref{r:wrtL}).
Let $\F$ denote the indefinite packing $\R^*$ integral of $\Div \mathbf u$ in $A$. 
Since $\F$ is the flux of $\mathbf u$ in $A$, we have $\F(A)=\int_{\partial_*A}\mathbf u\cdot\boldsymbol\nu_A\,d\Hnmj$.
By Theorem \ref{thm:admisn} we obtain 
$\int_{\cl_*A}\Div \mathbf u(x)\,dx=\F(A)$, which completes the proof.
\end{proof}

\section{$\R$ integral}

In this section we will introduce Pfeffer's $\R$ integral described in \cite{Pf}.
For easier comparison of integrals we use 
 the characterization of $\R$ integral \cite[Proposition 5.5.6]{Pf} rather than original definition.

\begin{definition}
A \emph{$\BV$ partition} is a system of couples $\{(A_1,x_1),\ldots,(A_k,x_k)\}$
of pairwise disjoint bounded $\BV$ sets $A_i$ and points $x_i\in\rn$ for $i=1,\ldots,k$.
It is not required $x_i\in A_i$.
\end{definition}

\begin{definition}\label{def:pff}
Let $A$ be a locally $\BV$ set and $\F$, $\G$ be charges in $A$.
Let $f$ be a function defined on $\cl_* A$.
We say that $\F$ 
is an \emph{intrinsic indefinite $\R$ integral} of $f$ in $A$ with respect to $\G$ if
for given $\varepsilon>0$ we can find a gage $\delta:\cl_* A\to[0,\infty)$
so that
$$
\sum_{i=1}^k \big|
\F(A_i)-f(x_i)\G(A_i)
\big|<\varepsilon
$$
for each $\varepsilon$-regular $\delta$-fine $\BV$ partition
$\{(A_1,x_1),\ldots,(A_k,x_k)\}$ with 
$\bigcup_{i=1}^k A_i\subset A$ and $x_i\in \cl_* A$ for $i=1,\ldots, k$.

The family of all $\R$ integrable functions in $A$ with respect to $\G$ is denoted by $\R(A,\G)$.
The family of all $\R$ integrable functions in $A$ with respect to Lebesgue measure is denoted just by $\R(A)$.
\end{definition}

\begin{remark}
The intrinsic indefinite $\R$ integral is well defined, unique and  linear. For the proof and other properties see
\cite[p. 211-213]{Pf}.
\end{remark}

\begin{definition}\label{def:pffj}
Let $A$ be a locally $\BV$ set and $\F$, $\G$ be charges in $A$.
Let $f$ be a function defined on $\cl_* A$.
We say that $\F$ 
is an \emph{indefinite  $\R$ integral} of $f$ in $A$ with respect to $\G$ if
for given $\varepsilon>0$ we can find a gage $\delta:\cl_* A\to[0,\infty)$
so that
$$
\sum_{i=1}^k \big|
\F(A_i)-f(x_i)\G(A_i)
\big|<\varepsilon
$$
for each $\varepsilon$-regular $\delta$-fine $\BV$ partition
$\{(A_1,x_1),\ldots,(A_k,x_k)\}$ with 
$x_i\in \cl_* A$ for $i=1,\ldots, k$. 
(We do not require that $A_i\subset A$.)

The family of all $\R$ integrable functions in $A$ with respect to $\G$ is denoted by $\I\R(A,\G)$.
\end{definition}

\begin{remark}
Let us remark that our terminology slightly differs from that used in \cite{Pf}. 
Namely, what we call ``intrinsic indefinite $\R$ integral in $A$'' is termed simply ``indefinite $\R$
integral'' in \cite{Pf}.
Furthermore, in \cite{Pf} it is distinguished between the $\R$ integral (with respect to Lebesgue measure)
and $\S$ integral (Stieltjes version;  with respect to an arbitrary charge).
\end{remark}

\begin{lemma}\label{l:krit}
Let $\varepsilon>0$ and $A\subset \rn$ be an $\varepsilon$-regular bounded $\BV$ set.
Then $[\diam (A)]^n\leq\frac{1}{\varepsilon^n}c |A|$, where $c=c(n)$ is a constant depending only on $n$.
\end{lemma}
\begin{proof}
Since $A$ is  $\varepsilon$-regular, we have $\diam (A)P(A)\leq \frac{1}{\varepsilon}|A|$.
Further, by the isoperimetric inequality (see \cite[Theorem 1.8.7]{Pf}) we have
$|A|^{\frac{n-1}{n}}\leq p(n)P(A)$, where $p(n)$ is a constant depending on $n$.
Thus $\diam (A)P(A)\leq \frac{1}{\varepsilon}|A|\leq \frac{1}{\varepsilon}p(n)^{\frac{n}{n-1}}P(A)^{\frac{n}{n-1}}$.
Hence
$$
\aligned
\diam (A)^{n-1}P(A)^{n-1}&\leq \frac{1}{\varepsilon^{n-1}}p(n)^nP(A)^{{n}},\\
\diam (A)^{n-1}&\leq \frac{1}{\varepsilon^{n-1}}p(n)^nP(A),\\
\diam (A)^{n}&\leq \frac{1}{\varepsilon^{n-1}}p(n)^nP(A)\diam (A)\\
&\leq \frac{1}{\varepsilon^{n}}c|A|,\\
\endaligned
$$
where $c=p(n)^n$.
\end{proof}

\begin{theorem}\label{thm:rinnein}
Let $A$ be a locally $\BV$ set, $\F$, $\G$ be charges in $A$. Let $f$ be a function defined on $\cl_* A$.
If $\F$ is an intrinsic indefinite $\R$ integral of $f$ in $A$ with respect to $\G$,
then $\F$ is also an indefinite $\R$ integral of $f$ in $A$ with respect to $\G$. 
\end{theorem}

\begin{proof}
Let us choose $\varepsilon\in(0,1/(c\alpha_n))$
and set $\varepsilon'=\varepsilon(1-c\alpha_n\varepsilon)/(1+c)$,
where $c=c(n)$ is as in Lemma \ref{l:krit}.
Let us find a gage  $\delta_1:\cl_* A\to[0,\infty)$
so that
$$
\sum_{i=1}^k \big|
\F(A_i)-f(x_i)\G(A_i)
\big|<\varepsilon'
$$
for each $\varepsilon'$-regular $\delta_1$-fine $\BV$ partition
$\{(A_1,x_1),\ldots,(A_k,x_k)\}$ with 
$\bigcup_{i=1}^k A_i\subset A$ and $x_i\in \cl_* A$ for $i=1,\ldots, k$.

Further, for each $x\in\intt_c A$ let us find $R=R(x)>0$ such that 
for every $r<R$ we have
\begin{equation}\label{e:krit}
P(A,B(x,r))\leq \varepsilon^{n-1} r^{n-1}.
\end{equation}

Since $\intt_c A\subset \intt_*A$, for every $x\in\intt_cA$ we can find $R'=R'(x)>0$ such that 
\begin{equation}\label{e:hust}
|B\setminus A|<\varepsilon^{n+1}|B|
\end{equation}
for every $B=B(x,r)$, $r<R'$.

Now let us define
$$
\delta(x)=\begin{cases}
0 &\text{if } x\in \cl_*A\setminus \intt_cA,\\
\min\{\delta_1(x),R(x),R'(x)\} &\text{if } x\in \intt_cA.\\
\end{cases}
$$
Since the set $\cl_*A\setminus \intt_c A$ is of $\sigma$-finite $\Hnmj$ Hausdorff measure (see \cite[Theorem 1.8.2]{Pf} and \cite[Proposition 7.3.1]{Pffin}), 
$\delta$ defines a gage.

Let us fix an 
$\varepsilon$-regular $\delta$-fine $\BV$ partition 
$\{(A_1,x_1),\ldots,(A_k,x_k)\}$ such that 
$x_i\in \intt_c A$ for $i=1,\ldots, k$.
We need to show that 
$$
\sum_{i=1}^k \big|
\F(A_i)-f(x_i)\G(A_i)
\big|<\varepsilon.
$$

Let us set $A'_i=A_i\cap A$, $i=1,\ldots, k$.
Then $\{(A'_1,x_1),\ldots,(A'_k,x_k)\}$ is obviously a $\delta$-fine $\BV$ partition.
We need to show that the system $\{(A'_1,x_1),\ldots,(A'_k,x_k)\}$ is $\varepsilon'$-regular.

Let us fix $i\in\{1,\ldots,k\}$.
Since $(A_i,x_i)$ is $\varepsilon$-regular, we have 
$$\diam(A_i\cup\{x_i\})P(A_i)\leq\frac{1}{\varepsilon}|A_i|.$$
Further, let us find a minimal ball ${B}={B}(x_i,r)$ with the property that $A_i\subset \bar{B}$.
Then $r<\delta(x)$. 

By Lemma \ref{l:krit} there exists a constant $c$ such that
$$|B|\leq \alpha_n \left(\diam (A_i\cup\{x_i\})\right)^n \leq \frac{c\alpha_n}{\varepsilon^n}|A_i|.$$
Then applying \eqref{e:hust} we can estimate
$$
\aligned
|A_i|&\leq |A_i\cap A|+|A_i\setminus A|\\
&\leq |A_i\cap A|+|B\setminus A|\\
&\leq |A_i\cap A|+\varepsilon^{n+1}|B|\\
&\leq |A_i\cap A|+c\alpha_n\varepsilon |A_i|.
\endaligned
$$
Hence
\begin{equation}\label{e:deleni}
(1-c\alpha_n\varepsilon)|A_i|\leq |A_i\cap A|.
\end{equation}

Applying \eqref{e:krit}, \eqref{e:deleni} and Lemma \ref{l:krit} then gives
$$
\aligned
\diam(A_i\cap A\cup\{x_i\})P(A_i\cap A)&\leq
\diam(A_i\cup\{x_i\})\left[P(A_i)+ P(A, {B})\right]\\
&\leq
\frac{1}{\varepsilon}|A_i|+\diam (A_i\cup\{x_i\})\varepsilon^{n-1} r^{n-1}\\
&\leq
\frac{1}{\varepsilon}|A_i|+\varepsilon^{n-1}[\diam (A_i\cup\{x_i\})]^{n}\\
&\leq
\left(\frac{1}{\varepsilon}+\frac{c}{\varepsilon}\right)
|A_i|\\
&\leq
\frac{1+c}{\varepsilon(1-c\alpha_n\varepsilon)}
|A_i\cap A|\\
&=
\frac{1}{\varepsilon'}
|A_i\cap A|.
\endaligned
$$

Thus the system $\{(A'_1,x_1),\ldots,(A'_k,x_k)\}$ is $\delta$-fine $\varepsilon'$-regular $\BV$ partition.
Since $\F$ and $\G$ are charges in $A$, we have $\F(A_i)=\F(A'_i)$ and $\G(A_i)=\G(A'_i)$ for $i=1,\ldots,k$. 
Further, since $\F$ is the intrinsic indefinite $\R$ integral of $f$ with respect to $\G$, we can estimate 
$$
\aligned
&\sum_{i=1}^k \big|
\F(A_i)-f(x_i)\G(A_i)
\big|
=
\sum_{i=1}^k \big|
\F(A'_i)-f(x_i)\G(A'_i)
\big|
<\varepsilon'<\varepsilon,
\endaligned
$$
which completes the proof.

\end{proof}
\begin{corollary}\label{c:rir}
Let $A$ be a locally $\BV$ set, $\F$, $\G$ be charges in $A$. Let $f$ be a function defined on $\cl_* A$.
Then $\F$ is an intrinsic indefinite $\R$ integral of $f$ in $A$ with respect to $\G$
if and only if $\F$ is an indefinite $\R$ integral of $f$ in $A$ with respect to $\G$. 
\end{corollary}

\begin{theorem}\label{thm:pr}
Let $A$ be an admissible locally $\BV$ set, $\F$, $\G$ be charges in $A$. Let $f$ be a function defined on $\cl_* A$.
Let $\F$ be an (intrinsic) indefinite $\R$ integral of $f$ in $A$  with respect to $\G$.
Then $\F$ is also an indefinite packing $\R$ integral of $f$  in $A$ with respect to $\G$.
\end{theorem}

\begin{proof}
Let us set $\tau:=1$. Now let us choose $\varepsilon>0$ and
 find a gage $\delta$ as in Definition \ref{def:pffj}. 
Let us fix a $\delta$-fine packing 
$\sys{B(x_i,r_i)}_{i=1}^k$, $x_i\in \cl_* A$.
We need to show that
$$
\sum_{i=1}^k
\bar{p}^\varepsilon_{x_i,\tau r_i}(\F-f(x_i)\G)<\varepsilon,
$$
where
$\bar{p}_{x,r}^\varepsilon(\F)=
\sup\left\{
|\F(E)|; E\subset\subset B(x,r), E\in \tbv, (E,x) \mbox{ is $\varepsilon$-regular}
\right\}.
$

Now, let us fix test sets $E_i$
such that 
$E_i\subset\subset B(x_i,r_i)$,
$E_i$ are $\BV$ sets
and $(E_i,x_i)$ are $\varepsilon$-regular for $i=1,\ldots,k$.
Obviously, the system $\{ (E_i,x_i)\}$ is  $\varepsilon$-regular $\delta$-fine $BV$ partition and hence by Definition
\ref{def:pffj} and Theorem \ref{thm:rinnein}
we have
$$
\sum_{i=1}^k
\left|\F(E_i)-f(x_i)\G(E_i)\right|
<\varepsilon.
$$

Passing to the supremum we obtain
$$
\sum_{i=1}^k
\bar{p}^\varepsilon_{x_i,\tau r_i}(\F-f(x_i)\G)\leq\varepsilon.
$$

\end{proof}

\section{$\GR$ integral}\label{s:sedm}

It can happen that a function which is $\R$ integrable in sets $A_1$ and 
$A_2$ is not $\R$ integrable in their union. Also, $\R$ integrability 
is not closed with respect to $\BV$ convergence of sets.
To correct this deficiency, Pfeffer \cite{Pf} extended the definition 
of the $\R$ integral.
Fortunately, the construction based on the closure with respect to 
$\BV$ convergence of sets solves automatically the problem of 
additivity. The result of this construction is called $\GR$ integral 
(the generalized Riemann integral).
Using our Theorem \ref{thm:sonmn} we show that also this $\GR$ integral 
is contained in our packing $\R^*$ integral.

\begin{notation}
Let $f$ be a function whose domain contains a locally $\BV$ set $E$ and let $\F$ be a charge. 

Then we denote by $\tR(f,\F,E)$ the family of all bounded $\BV$ sets $A\subset E$ such that
$f\chi_A$ belongs to $\R(A)$ and the charge $\F\lfloor_A$ is the indefinite $\R$ integral of $f\chi_A$.

Further, let us denote 
$\tpR(f,\F,E)$ the minimal system of bounded $BV$ sets containing $\tR(f,\F,E)$ 
and closed with respect to convergence in $\tbv$.
\end{notation}

\begin{definition}
Let $f$ be a function  defined on a locally $\BV$ set $E$.
We say that a charge $\F$ is an 
\emph{indefinite $\GR$ integral} of $f$ in $E$ if
$\tpR(f,\F,E)=\tbv(E)$,
where $\tbv(E)=\{A\in\tbv;A\subset E\}$ .
The family of all $\GR$ integrable functions in $E$ is denoted by $\GR(E)$.
\end{definition}

\begin{remark}
The indefinite $\GR$ integral is well defined, unique and linear. For further details see \cite[Sec. 6.3]{Pf}.
\end{remark}

The next theorem with proof can be found in \cite[Proposition 6.3.12]{Pf}.
\begin{theorem}\label{thm:RGR}
Let $E$ be a locally $\BV$ set. Then 
\begin{enumerate}
\item\label{thm:RGRj} If $n=1$, then $\R(E)=\GR(E)$. 
\item\label{thm:RGRd} If $n\geq 2$ and $\intt E\neq \emptyset$, then $\R(E)\subsetneq \GR(E)$.
\end{enumerate}
\end{theorem}

\begin{theorem}\label{thm:sonmn}
Let $E$ be a bounded admissible $BV$ set. Then $\GR(E)\subset\P\R^*(E)$.
\end{theorem}
\begin{proof}
The proof follows from Theorems \ref{thm:rhr}, \ref{thm:pr}  and \ref{thm:conmn}.
\end{proof}


\section{$\R^*$ integral}
The $\R^*$ integral  was introduced by Mal\'{y} and Pfeffer in \cite{mpf}.
It is an alternative approach to overcome drawbacks of the $\R$ integral.
Moreover, in $\er^1$ this integral coincides with
the Henstock-Kurzweil integral.

\begin{definition}
Let $\varepsilon>0$.
We say that an $\varepsilon$-regular $\BV$ partition $\{(A_1,x_1),\ldots,\penalty 0 
 (A_k,x_k)\}$ is \emph{strongly $\varepsilon$-regular} if 
$A_i$ is $\varepsilon$-isoperimetric and $x_i\in\cl_* A_i$ for $i=1,\ldots, k$.
\end{definition}

\begin{definition}\label{def:mp}
Let $A\subset \rn$ be a locally $\BV$ set. 
We say that a charge $\F$ in $A$ is an \emph{indefinite $\R^*$ integral} of a function $f:\cl_* A\to\er$ in $A$
 with respect to a charge $\G$ if 
for given $\varepsilon>0$ we can find a gage $\delta:\cl_* A\to[0,\infty)$ 
so that
$$
\sum_{i=1}^k \big|
\F(A_i)-f(x_i)\G(A_i)
\big|<\varepsilon
$$
for each strongly $\varepsilon$-regular $\delta$-fine $\BV$ partition
$\{(A_1,x_1),\ldots,(A_k,x_k)\}$.

The family of all $\R^*$ integrable functions in $A$ is denoted by $\R^*(A,\G)$.
The family of all $\R^*$ integrable functions in $A$ with respect to Lebesgue measure is denoted just by $\R^*(A)$.
\end{definition}

\begin{remark}
It is easily seen that for an admissible $\BV$ set $E$ we have $\R(E)\subset \R^*(E)$.
\end{remark}

\begin{theorem}\label{thm:mp}
Let $A\subset\rn$ be a locally $\BV$ set.
Let a charge $\F$ be an indefinite $\R^*$ integral of a function $f:\cl_* A\to\er$ in $A$ with respect to a charge $\G$.
Then $\F$ is also an indefinite packing $\R^*$ integral of $f$ in $A$ with respect to $\G$.
\end{theorem}
\begin{proof}
Let us set $\tau:=1$. Then let us choose $\varepsilon>0$ and
find a gage $\delta$ as in Definition \ref{def:mp}.

Let us fix a $\delta$-fine packing 
$\sys{B(x_i,r_i)}_{i=1}^k$, $x_i\in \cl_*A$.
We need to show that
$$
\sum_{i=1}^k
\bar{q}^\varepsilon_{x_i,\tau r_i}(\F-f(x_i)\G)<\varepsilon,
$$
where
$$
\aligned
\bar{q}_{x,r}^\varepsilon(\F)&=
\sup\{
|\F(E)|; E\subset\subset B(x,r), E\in \tbv, x\in\cl_*E, (E,x) 
\mbox{ is }\varepsilon\text{-regular} \\
&\qquad \mbox{ and }\varepsilon\text{-isoperimetric}
\}.
\endaligned
$$

Now let us fix test sets $E_i$, $E_i\subset\subset B(x_i,r_i)$, $x_i\in\cl_*E_i$,
  $E_i$ is $\BV$ and $(E,x)$ is $\varepsilon$-regular and $\varepsilon$-isoperimetric for $i=1,\ldots, k$.

Obviously, the system $\{ (E_i,x_i)\}$ is  strongly $\varepsilon$-regular $\delta$-fine $\BV$ partition
 and hence by Definition \ref{def:mp}
we obtain
$$
\sum_{i=1}^k
\left|\F(E_i)-f(x_i)\G(E_i)\right|
<\varepsilon.
$$

Passing to the supremum we obtain
$$
\sum_{i=1}^k
\bar{q}^\varepsilon_{x_i,\tau r_i}(\F-f(x_i)\G)\leq\varepsilon.
$$

\end{proof}

\begin{remark}\label{r:rhgr}
Let $E$ be an admissible locally 
$BV$ set. Then $\GR(E)\subsetneq \R^*(E)$. 
The inclusion follows from \cite[Corollary 3.18]{mpf} and 
\cite[Theorem 3.20]{mpf}.
An example of function which is $\R^*$ integrable but not $\GR$ integrable can be found in 
\cite[Example 6.9]{Pf-cl} and \cite[Proposition 10.8]{Pf-cl}.
\end{remark}

\section{Henstock-Kurzweil-Stieltjes integral}
In the next two sections we will investigate packing $\R$ and packing $\R^*$ integral on the real line.
For this purpose let us note that a charge $\F$ in $\er^1$ can be identified
  with an ``ordinary'' function $F$ through the relation
$\F\left((u,v)\right)=F(v)-F(u)$.
Since for those integral $\F$ is supposed to be a charge, in the view of Example \ref{r:ident}, $F$ is continuous.

\begin{definition}
Let $[a,b]\subset\er^1$ be a compact interval.
A finite collection\linebreak $\left( [a_i,b_i], \xi_i\right)_{i=1}^k$ of tagged intervals is called a \emph{subpartition} of $[a,b]$
if intervals $[a_i,b_i]$ are nonoverlapping and $\xi_i\in[a_i,b_i]$ for every $i=1,\ldots,k$.


A function $\delta:[a,b]\to(0,\infty)$ is called a \emph{positive gage}.
We say that a subpartition is $\delta$-fine if $|b_i-a_i|<\delta(\xi_i)$.
\end{definition}

\begin{definition}
Let $f,G,F:[a,b]\to\er$ be functions. 
We say that  $F$ is the \emph{strong Henstock-Kurzweil-Stieltjes integral} of $f$ with respect to $G$ if 
for every $\varepsilon>0$, there exists a positive gage $\delta:[a,b]\to(0,\infty)$, so that
for every $\delta$-fine subpartition $\left( [a_i,b_i], \xi_i\right)_{i=1}^k$ we have
$$
\sum_{i=1}^k \left|F(b_i)-F(a_i)-f(\xi_i)(G(b_{i})-G(a_{i}))\right|<\varepsilon.
$$
In the case $G$ is the identity function we say that $F$ is just the \emph{strong Henstock-Kurzweil integral} of $f$.

The families of all strongly Henstock-Kurzweil-Stieltjes  integrable functions on $[a,b]$ with respect to $G$
and all strongly Henstock-Kurzweil integrable functions on $[a,b]$ 
are denoted by $HKS([a,b],G)$ and $HK([a,b])$, respectively.
\end{definition}

\begin{definition}
Let $f,G,F:\er\to\er$ be functions. 
We say that  $F$ is the \emph{indefinite Henstock-Kurzweil-Stieltjes integral} of $f$ with respect to $G$ if 
$F$ is the strong Henstock-Kurzweil-Stieltjes integral of $f$ with respect to $G$ on every compact interval $[a,b]\subset\er$.

In the case $G$ is the identity function we say that $F$ is just the \emph{indefinite Henstock-Kurzweil integral} of $f$.

The families of all Henstock-Kurzweil-Stieltjes  integrable functions with respect to $G$
and all Henstock-Kurzweil integrable functions 
are denoted by $HKS(\er,G)$ and $HK(\er)$, respectively.
\end{definition}

The proof of the following proposition can be found in \cite[Proposition 3.6]{mpf}.
\begin{proposition}\label{p:RhHK}
Let $f:\er\to\er$ be a function. Then $f$ is $\R^*$ integrable with respect to Lebesgue measure
if and only if $f$ is strongly Henstock-Kurzweil integrable on every compact interval $[a,b]\subset\er$.
\end{proposition}

Applying Theorem \ref{thm:RGR}, Remark \ref{r:rhgr} and Proposition \ref{p:RhHK} we obtain the following theorem.
\begin{theorem}\label{t:RHK}
$\R(\er)\varsubsetneq HK(\er)$.
\end{theorem}

\section{$MC$ and $MC_\alpha$ integrals}

In this section we will introduce $MC$ and $MC_\alpha$ integrals.
The monotonically controlled Stieltjes ($MC$) integral was defined by Bendov\'{a} and Mal\'{y} in \cite{BM}.
The theory of the $MC_\alpha$ integral with respect to Lebesgue measure was further developed by 
Ball and Preiss in \cite{dball}. Their ideas will be used in the proofs of
 Propositions \ref{p:mccont} and \ref{p:albe}.

\begin{definition}\label{d:amc}
Let $\alpha>0$ be a real number and $f,F,G:\er\to\er$ be functions, let $G$ be continuous.
We say that $F$ is an \emph{indefinite $MC_\alpha$ integral} of $f$ with respect to $G$ if there exists a strictly increasing control function
$\varphi:\er\to\er$ such that for each $x\in \er$ we have
$$
\lim_{h\to 0}\frac{F(x+h)-F(x)-f(x)(G(x+h)-G(x))}{\varphi(x+\alpha h)-\varphi(x)}=0.
$$

The families of all $MC_\alpha$ integrable functions with respect to $G$ 
and all $MC_\alpha$ integrable functions with respect to identity function 
are denoted by 
$MC_\alpha(G)$ and $MC_\alpha$, respectively.

Especially, if $\alpha=1$, we say that $F$ is an \emph{indefinite $MC$ integral of $f$ with respect to $G$}.
We write $MC(G)=MC_1(G)$ and  $MC=MC_1$.
\end{definition}

\begin{remark}\label{r:bfi}
In Definition \ref{d:amc} the control function $\varphi$ can be chosen to be bounded. (See \cite[Lemma 1]{BM}.)
\end{remark}

\begin{proposition}\label{p:mccont}
Let $\alpha>0$  and let $f,F,G:\er\to\er$ be functions, let $G$ be continuous.
If $F$ is an indefinite $MC_\alpha$ integral of $f$ with respect to $G$, then $F$ is continuous.
\end{proposition}
\begin{proof}
Let us fix $\varepsilon>0$ and $x\in \er$.
We need to find $\delta$ such that for every $|h|<\delta$ we have
$$
|F(x+h)-F(x)|<\varepsilon.
$$

 Since $G$ is continuous at $x$, we can find $\delta_1$ such that
for every $|h|<\delta_1$ we have $|G(x+h)-G(x)|<\varepsilon$.
Further, since $F$ is the indefinite $MC_\alpha$ integral of $f$ with respect to $G$, there exists a strictly increasing control function
$\varphi:\er\to\er$ and a $\delta<\delta_1$ such that for every $|h|<\delta$ we have
$$
\left|
\frac{F(x+h)-F(x)-f(x)(G(x+h)-G(x))}{\varphi(x+\alpha h)-\varphi(x)}
\right|<\varepsilon.
$$
Applying Remark \ref{r:bfi} we can assume that there exists a constant $M$ such that $|\varphi(x)|<M$ for every $x\in \er$.

Hence 
$$\aligned
&|F(x+h)-F(x)|\\
&\qquad
\leq
\left|\frac{F(x+h)-F(x)-f(x)(G(x+h)-G(x))}{\varphi(x+\alpha h)-\varphi(x)}(\varphi(x+\alpha
h)-\varphi(x))\right|\\
&\qquad\quad+\left|f(x)(G(x+h)-G(x))
\right|\\
&\qquad<\varepsilon (2M+f(x)).
\endaligned
$$

\end{proof}

\begin{proposition}\label{p:albe}
Let $0<\alpha<\beta$ be real numbers,
$f,F,G:\er\to\er$ be functions and let $G$ be continuous.
If $F$ is an indefinite $MC_\alpha$ integral of $f$ with respect to $G$, then
$F$ is also an indefinite $MC_\beta$ integral of $f$ with respect to $G$.
\end{proposition}
\begin{proof}
The proof follows from the fact that  for $0<\alpha<\beta$ we have
$|\varphi(x+\alpha h)-\varphi(x)|\leq
|\varphi(x+\beta h)-\varphi(x)|$ for $h\in\er$.
\end{proof}

The two following theorems can be found in \cite[Theorem 3]{dball}.
\begin{theorem}\label{thm:pmca}
For every $\alpha\geq2$ there exists a function which is not $MC_\alpha$ integrable but
is $MC_\beta$ integrable for every $\beta>\alpha$.
\end{theorem}
\begin{theorem}\label{thm:mcmca}
Let $\alpha>2$. Then $MC$ is a proper subspace of $MC_\alpha$.
\end{theorem}

For the proof of the next theorem see
\cite[Theorem 3]{dball}.
\begin{theorem}\label{thm:mcmci}
Let $\alpha\in[1,2]$. Then $MC=MC_\alpha$.
\end{theorem}

\begin{theorem}\label{thm:mchks}
Let $G,F,f:\er\to\er$ be functions. 
Suppose that $G$ is continuous.
Then $F$ is an indefinite $MC$ integral of $f$ with respect to $G$ if and only if
$F$ is an indefinite Henstock-Kurzweil-Stieltjes integral of $f$ with respect to $G$. 
\end{theorem}
\begin{proof}
For the proof and further details see \cite[Theorem 3]{BM} and \cite[Theorem 17]{dball}.
\end{proof}

\begin{theorem}\label{thm:pmc}
Let $\alpha\geq 1$, $G,F,f:\er\to\er$ be functions.
Suppose that $G$ is continuous.
Let $F$ be an indefinite $MC_\alpha$ integral of $f$ with respect to $G$. 
Further, let $\F$ and $\G$ be charges induced by $F$ and $G$ in the sense of Example \ref{r:ident}.
Then $\F$ is also an indefinite packing $\R$ integral of $f$ with respect to $\G$.
\end{theorem}

\begin{proof}
First, let us note that $F$ is continuous by Proposition \ref{p:mccont}. Hence it is legitimate to 
use the term charges for the set functions $\F$ and $\G$ constructed as in Example \ref{r:ident}.

 Let us set $\tau:=1/\alpha$.
 Further,  let us fix $\varepsilon>0$ and write $\varepsilon':=\varepsilon^2$.
Since $f$ is $MC_\alpha$ integrable, there exists a strictly increasing function $\varphi:\er\to\er$ with the following
property:
for each $x\in \er$ there exists $\delta(x)>0$ such that for every $|h|<\delta(x)$ we have
\begin{equation}\label{e:fi}
|F(x+h)-F(x)-f(x)(G(x+h)-G(x))|<\varepsilon'|\varphi(x+\alpha h)-\varphi(x)|.
\end{equation}
Moreover, by Remark \ref{r:bfi} we can suppose that there exists $M>0$ such that $|\varphi|\leq M$.

We need to show that for fixed $\delta$-fine packing 
$\sys{B(x_i,r_i)}_{i=1}^k$, we have
$$
\sum_{i=1}^k
\bar{p}^{\varepsilon}_{x_i,\tau r_i}(\F-f(x_i)\G)<\varepsilon,
$$
where
$
\bar{p}_{x_i,\tau r_i}^{\varepsilon}(\F)=
\sup\left\{
|\F(E)|; E\subset\subset B(x_i,\tau r_i), E\in \tbv, (E,x_i) \mbox{ is $\varepsilon$-regular}
\right\}.
$

Let us fix $i\in\{1,\ldots, k\}$ and a test set 
$E_i\in\tbv$ such that $E_i\subset\subset B(x_i,\tau r_i)$ and $(E_i,x_i)$ is $\varepsilon$-regular.
In other words, $E_i=\bigcup_{j=1}^{l_i} (a^i_j,b^i_j)$ is a finite union of disjoint nondegenerate intervals in
$B(x_i,\tau r_i)$
(up to a Lebesgue null set). 
Moreover, since
$E_i$ is $\varepsilon$-regular
and $\H^0$ is the counting measure, we estimate
$$\frac{1}{\|E_i\|}\geq
\frac{|E_i|}{d(E_i\cup \{x_i\})\|E_i\|}>\varepsilon$$
and
$$
\frac{1}{\varepsilon}\geq \|E_i\|
=\H^0(\partial_*E_i)=2l_i.
$$
Let us set $m$ to be the greatest natural number such that $m\leq 1/(2\varepsilon)$. Then $l_i\leq m\leq 1/(2\varepsilon)$.

Further, since for each $a_j^i$ and $b_j^i$, $i=1,\ldots,k$ and $j=1,\ldots,l_i$, we have $|a_j^i-x_i|<\delta(x_i)$ and
$|b_j^i-x_i|<\delta(x_i)$, by \eqref{e:fi} and the fact that $\varphi$ is increasing we have the estimates
$$
\left|F\left(b_j^i\right)-\F(x_i)-f(x_i)(G(b_j^i)-G(x_i))\right|
<\varepsilon'\left|\varphi\left(x_i+{r_i}\right)-\varphi(x_i-r_i)\right|
$$
and 
$$
\left|F(a_j^i)-\F(x_i)-f(x_i)(G(a_j^i)-G(x_i))\right|
<\varepsilon'\left|\varphi\left(x_i+{r_i}\right)-\varphi(x_i-r_i)\right|.
$$
 
Moreover, since the system $(B(x_i,r_i))_{i=1}^k$ is pairwise disjoint and $\varphi$ is strictly increasing and bounded, we have 
\begin{equation}\label{e:fik}
\aligned
&\sum_{i=1}^k 
\left| F(b^i_j)-F(a^i_j)-f(x_i)(G(b^i_j)-G(a^i_j))\right|\\
&\qquad \leq \sum_{i=1}^k 
\left|F(b_j^i)-\F(x_i)-f(x_i)(G(b_j^i)-G(x_i))\right|\\
&\quad\qquad +\left|F(a_j^i)-\F(x_i)-f(x_i)(G(a_j^i)-G(x_i))\right|\\
&\qquad\leq \sum_{i=1}^k 
2\varepsilon'\left|\varphi\left(x_i+r_i\right)-\varphi(x_i-r_i)\right|\\
&\qquad<2\varepsilon'(\varphi(x_k+r_k)-\varphi(x_1-r_1))\\
&\qquad<
4\varepsilon'M.
\endaligned
\end{equation}

Let us denote $L:=\max_i l_i$. 
For $j\in\{1,\ldots,L\}$ let $I_j$ be the set of indices $i\in\{1,\ldots,k\}$ for which
$l_i\geq j$.
Then applying estimates in \eqref{e:fik} we obtain
$$
\aligned
 \sum_{i=1}^k 
\left|\F(E_i)-f(x_i)\G(E_i)\right|&\leq
\sum_{i=1}^k \sum_{j=1}^{l_i}
\left| F(b^i_j)-F(a^i_j)-f(x_i)(G(b^i_j)-G(a^i_j))\right|\\
&\leq \sum_{i=1}^k \sum_{j=1}^{l_i}
\left|F(b_j^i)-\F(x_i)-f(x_i)(G(b_j^i)-G(x_i))\right|\\
&\quad+
\left|F(a_j^i)-\F(x_i)-f(x_i)(G(a_j^i)-G(x_i))\right|\\
&\leq \sum_{j=1}^L \sum_{i\in I_j}
\left|F(b_j^i)-\F(x_i)-f(x_i)(G(b_j^i)-G(x_i))\right|\\
&\quad +
\left|F(a_j^i)-\F(x_i)-f(x_i)(G(a_j^i)-G(x_i))\right|\\
&<\sum_{j=1}^{L}4\varepsilon'M=4L\varepsilon^2M
\leq \frac{4\varepsilon^2M}{2\varepsilon}
=2M\varepsilon.
\endaligned
$$

Finally, passing to the supremum we obtain
$$
\sum_{i=1}^k
\bar{p}^\varepsilon_{x_i, r_i}(\F-f(x_i)\G)\leq 2M\varepsilon,
$$
which we needed.

\end{proof}

\section{Summary of relations}

Let $A\subset \rn$ be an admissible locally $\BV$ set.
The relation between classes of integrable functions in $A$ is shown in the following diagram.

$$
\begin{tikzcd}
    \I\R \arrow[draw=none]{r}[sloped,auto=false]{\subsetneq} \arrow[draw=none]{d}[sloped,auto=false]{=}& \G\R \arrow[draw=none]{r}[sloped,auto=false]{\subsetneq}& \R^*\arrow[draw=none]{d}[sloped,auto=false]{\subsetneq}\\
    \R \arrow[draw=none]{r}[sloped,auto=false]{\subsetneq} & \P\R \arrow[draw=none]{r}[sloped,auto=false]{\subset}& \P\R^*\\
  \end{tikzcd}
$$
The strictness of the inclusion $\I\R \subset \G\R$ holds for $n\ge 2$ and can be found in Theorem \ref{thm:RGR}(\ref{thm:RGRd}) and 
Corollary \ref{c:rir}; the case $n=1$ is discussed below.
The fact that $\G\R \subsetneq \R^*$ is mentioned in Remark \ref{r:rhgr}.
Corollary \ref{c:rir} shows the equality of $\I\R$ and $\R$.
The relationship $\R \subsetneq\P\R$ is described in Theorem \ref{thm:pr};
Theorems \ref{t:RHK}, \ref{thm:pmc} and \ref{thm:mchks} show that this inclusion is strict.
The inclusion $\P\R \subset \P\R^*$ is proved in Theorem \ref{thm:rhr}.
Theorem \ref{thm:mp} proves the inclusion $\R^*\subsetneq \P\R^*$,
the fact, that this inclusions is proper follows from Theorems \ref{thm:pmc}, \ref{thm:mchks}, \ref{thm:mcmca}  and Proposition \ref{p:RhHK}.

In the case $A=\er$, we can compare integrable functions in the following way.
$$
\begin{tikzcd}
    \G\R \arrow[draw=none]{r}[sloped,auto=false]{=}&
    \R \arrow[draw=none]{r}[sloped,auto=false]{\subsetneq}&
    HK \arrow[draw=none]{r}[sloped,auto=false]{=} 
\arrow[draw=none]{d}[sloped,auto=false]{=}& 
MC\arrow[draw=none]{d}[sloped,auto=false]{=}
\arrow[draw=none]{r}[sloped,auto=false]{\subsetneq}& 
MC_\beta \arrow[draw=none]{r}[sloped,auto=false]{\subsetneq}& 
\P\R \arrow[draw=none]{r}[sloped,auto=false]{\subset} &
\P\R^*
\\
    &&\R^* & MC_\alpha \\
  \end{tikzcd}
$$

The equality $\GR=\R$ is described in Theorem \ref{thm:RGR}(\ref{thm:RGRj})
and the inclusion $\R\subsetneq HK$  in Theorem \ref{t:RHK}.
The fact that $HK$ integral coincides with $\R^*$ integral can be found in \ref{p:RhHK}.
Theorem \ref{thm:mchks} shows the equality $HK=MC$.
Theorem \ref{thm:mcmci} proves the equality $MC=MC_\alpha$ for $\alpha\in[1,2]$.
The inclusion $MC\subsetneq MC_\beta$ for $\beta>2$ is proved in Proposition \ref{p:albe}, 
the fact, that this inclusion is proper is shown in Theorem \ref{thm:mcmca}. 
The relationship $MC_\beta\subsetneq\P\R$ (not only) for $\beta\geq2$ is proved in Theorem \ref{thm:pmc}
and \ref{thm:pmca}.
Finally, the inclusion $\P\R \subset \P\R^*$ is shown in Theorem \ref{thm:rhr}.


\section*{Acknowledgements}
The research was supported by the grants GA\,\v{C}R P201/15-08218S and P201/18-07996S.

The author would also like to express deep gratitude to 
Jan Mal\'{y} for many valuable suggestions and helpful comments.

\bibliographystyle{abbrv}

\begin{thebibliography}{10}

\bibitem{AFP}
L.~Ambrosio, N.~Fusco, and D.~Pallara.
\newblock {\em Functions of bounded variation and free discontinuity problems}.
\newblock Oxford Mathematical Monographs. The Clarendon Press, Oxford
  University Press, New York, 2000.

\bibitem{dball}
T.~Ball and D.~Preiss.
\newblock Monotonically controlled integrals.
\newblock The university of Warwick, Mathematics Institute, 2016.

\bibitem{BM}
H.~Bendov\'a and J.~Mal\'y.
\newblock An elementary way to introduce a {P}erron-like integral.
\newblock {\em Ann. Acad. Sci. Fenn. Math.}, 36(1):153--164, 2011.

\bibitem{dpw}
Z.~Buczolich, T.~De~Pauw, and W.~F. Pfeffer.
\newblock Charges, {BV} functions, and multipliers for generalized {R}iemann
  integrals.
\newblock {\em Indiana Univ. Math. J.}, 48(4):1471--1511, 1999.

\bibitem{ZiemGG}
G.-Q. Chen, M.~Torres, and W.~P. Ziemer.
\newblock Gauss-{G}reen theorem for weakly differentiable vector fields, sets
  of finite perimeter, and balance laws.
\newblock {\em Comm. Pure Appl. Math.}, 62(2):242--304, 2009.

\bibitem{EG}
L.~C. Evans and R.~F. Gariepy.
\newblock {\em Measure theory and fine properties of functions}.
\newblock Studies in Advanced Mathematics. CRC Press, Boca Raton, FL, 1992.

\bibitem{He}
R.~Henstock.
\newblock Definitions of {R}iemann type of the variational integrals.
\newblock {\em Proc. London Math. Soc. (3)}, 11:402--418, 1961.

\bibitem{ph}
P.~Honz{\'{\i}}k and J.~Mal{\'y}.
\newblock Non-absolutely convergent integrals and singular integrals.
\newblock {\em Collect. Math.}, 65(3):367--377, 2014.

\bibitem{JKS}
J.~Jarn{\'{\i}}k, J.~Kurzweil, and v.~Schwabik.
\newblock On {M}awhin's approach to multiple nonabsolutely convergent integral.
\newblock {\em \v Casopis P\v est. Mat.}, 108(4):356--380, 1983.

\bibitem{KM}
K.~Kuncov{\'a} and J.~Mal{\'y}.
\newblock Non-absolutely convergent integrals in metric spaces.
\newblock {\em J. Math. Anal. Appl.}, 401(2):578--600, 2013.

\bibitem{mawhin}
J.~Kurzweil, J.~Mawhin, and W.~F. Pfeffer.
\newblock An integral defined by approximating {BV} partitions of unity.
\newblock {\em Czechoslovak Math. J.}, 41(116)(4):695--712, 1991.

\bibitem{malydistr}
J.~Mal\'y.
\newblock Non-absolutely convergent integrals with respect to distributions.
\newblock {\em Ann. Mat. Pura Appl. (4)}, 193(5):1457--1484, 2014.

\bibitem{mpf}
J.~Mal\'y and W.~F. Pfeffer.
\newblock Henstock-{K}urzweil integral on {BV} sets.
\newblock {\em Math. Bohem.}, 141(2):217--237, 2016.

\bibitem{Ma1}
J.~Mawhin.
\newblock Generalized multiple {P}erron integrals and the {G}reen-{G}oursat
  theorem for differentiable vector fields.
\newblock {\em Czechoslovak Math. J.}, 31(106)(4):614--632, 1981.

\bibitem{Pf-cl}
W.~F. Pfeffer.
\newblock The {G}auss-{G}reen theorem.
\newblock {\em Adv. Math.}, 87(1):93--147, 1991.

\bibitem{npf}
W.~F. Pfeffer.
\newblock {\em The {R}iemann approach to integration}, volume 109 of {\em
  Cambridge Tracts in Mathematics}.
\newblock Cambridge University Press, Cambridge, 1993.
\newblock Local geometric theory.

\bibitem{Pf}
W.~F. Pfeffer.
\newblock {\em Derivation and integration}.
\newblock Cambridge Tracts in Mathematics 140. Cambridge University Press,
  Cambridge, 2001.

\bibitem{Pffin}
W.~F. Pfeffer.
\newblock {\em The divergence theorem and sets of finite perimeter}.
\newblock Pure and Applied Mathematics (Boca Raton). CRC Press, Boca Raton, FL,
  2012.

\bibitem{Silh1}
M.~\v{S}ilhav\'y.
\newblock Divergence measure fields and {C}auchy's stress theorem.
\newblock {\em Rend. Sem. Mat. Univ. Padova}, 113:15--45, 2005.

\bibitem{Silh2}
M.~\v{S}ilhav\'y.
\newblock Cauchy's stress theorem for stresses represented by measures.
\newblock {\em Contin. Mech. Thermodyn.}, 20(2):75--96, 2008.

\bibitem{Silh3}
M.~\v{S}ilhav\'y.
\newblock The divergence theorem for divergence measure vectorfields on sets
  with fractal boundaries.
\newblock {\em Math. Mech. Solids}, 14(5):445--455, 2009.

\bibitem{ZiemC}
W.~P. Ziemer.
\newblock Cauchy flux and sets of finite perimeter.
\newblock {\em Arch. Rational Mech. Anal.}, 84(3):189--201, 1983.

\end{thebibliography}

\def\cprime{$'$} \def\cprime{$'$}

\end{document}